\documentclass[reqno]{amsart}

\usepackage{hyperref}
\usepackage{amssymb}
\usepackage{dsfont}
\usepackage[all]{xy}

\newcommand\supp[1]{\ensuremath{{\rm supp}\left(#1\right)}}
\def\card{{\#}}
\def\dist{{\rm dist}}
\def\Hdim{\dim_{\rm H}}
\def\Pdim{\dim_{\rm P}}
\def\hau{\mathcal{H}}
\def\leb{\mathcal{L}}
\def\ph{\varphi}
\def\eps{\varepsilon}
\def\osc{{\rm Osc}}
\newcommand{\lic}[2]{\ensuremath{\mathrm{G}^{#1}(#2)}}
\newcommand{\licfalc}[1]{\ensuremath{\mathcal{G}^{#1}}}
\newcommand{\varlic}[2]{\ensuremath{\overline{\mathrm{G}}^{#1}(#2)}}
\newcommand{\varlicfalc}[1]{\ensuremath{\overline{\mathcal{G}}^{#1}}}
\def\R{\mathbb{R}}
\def\N{\mathbb{N}}
\def\Z{\mathbb{Z}}

\def\C{\mathbb{C}}
\def\ZD{\Z^{d}_{\ast}}
\def\Ccal{\mathcal{C}}
\def\Fcal{\mathcal{F}}
\def\Mcal{\mathcal{M}}
\def\Ncal{\mathcal{N}}
\def\Ical{\mathcal{I}}
\def\Scal{\mathcal{S}}
\def\Hcal{\mathcal{H}}
\def\Jcal{\mathcal{J}}
\def\Pcal{\mathcal{P}}
\def\ind{\mathds{1}}
\def\sfrak{\mathfrak{s}}
\def\mfrak{\mathfrak{m}}
\def\Hfrak{\mathfrak{H}}
\newcommand{\diam}[1]{\ensuremath{\left| #1 \right|}}
\def\dd{\mathrm{d}}
\def\ee{\mathrm{e}}
\def\irrd{\Ical^{d}}
\newcommand{\prim}[1]{\ensuremath{\Pcal_{#1}}}
\newcommand{\disc}[1]{\ensuremath{D_{#1}}}
\def\Tau{\mathrm{T}}
\newcommand{\opball}[2]{\mathrm{B}(#1,#2)}

\theoremstyle{definition}
\newtheorem{df}{Definition}
\theoremstyle{plain}
\newtheorem{thm}{Theorem}
\newtheorem{prp}{Proposition}
\newtheorem{lem}{Lemma}
\newtheorem{cor}{Corollary}

\title{Multivariate Davenport series}
\author{Arnaud Durand \and St\'ephane Jaffard}
\address{Laboratoire de Math\'ematiques, UMR 8628\\ Universit\'e Paris-Sud\\ 91405 Orsay Cedex\\ France} 
\email{arnaud.durand@math.u-psud.fr}
\address{Laboratoire d'Analyse et de Math\'ematiques Appliqu\'ees, UMR 8050\\ Universit\'e Paris-est Cr\'eteil Val de Marne\\ 61 avenue du G\'en\'eral de Gaulle\\ 94010 Cr\'eteil Cedex\\ France}
\email{jaffard@u-pec.fr}

\subjclass[2010]{11J83, 11K55, 26A16, 28A78, 28A80, 42C15, 46E35}

\begin{document}

\begin{abstract}
We consider series of the form $\sum a_n \{n\cdot x\}$, where $n\in\Z^{d}$ and $\{x\}$ is the sawtooth function. They are the natural multivariate extension of Davenport series. Their global (Sobolev) and pointwise regularity are studied and their multifractal properties are derived. Finally, we list some open problems which concern the study of these series.
\end{abstract}

\maketitle


\section{Introduction}

Let $\lfloor\,\cdot\,\rfloor$ denote integer part and let $\{\,\cdot\,\}$ be the {\em centered sawtooth function} defined by
\begin{equation}
\{x\}=\begin{cases}
x-\lfloor x\rfloor -\frac{1}{2} & \mbox{if} \;\; x\notin\Z \\
0 & \mbox{else.}
\end{cases}
\end{equation}
The purpose of this paper is to investigate regularity properties of the multivariate functions which are defined by
\begin{equation}\label{eq:dave}
\forall x\in\R^{d} \qquad f(x)=\sum_{n\in\ZD} a_{n}\{n\cdot x\},
\end{equation}
where $n\cdot x$ denotes the standard inner product between the vectors $n$ and $x$, and $(a_{n})_{n\in\ZD}$ is a real valued sequence indexed by the set $\ZD=\Z^{d}\setminus\{0\}$. With a slight abuse, the vectors $n$ for which $a_{n}$ is nonvanishing will be referred to as the {\em frequencies} of the series.

In the one-variable case, examples of such functions can be traced back to the {\em Habilitationsschrift} of Riemann, see~\cite{Kahane:1996qv,Riemann:1868fk}; they were later considered by Hecke~\cite{Hecke:1921ly}, and also Hardy, who studied the series
\begin{equation}\label{eq:defHecke}
\Hfrak_{\beta}(x)=\sum_{n=1}^{\infty} \frac{\{nx\}}{n^{\beta}}.
\end{equation}
It seems however that the general one-dimensional case was first considered only in~1937 by H.~Davenport in~\cite{Davenport:1937yq,Davenport:1937vn}. The first of these papers starts with the following remarkable identity, which establishes in all generality the connexion with Fourier series:
\begin{equation}\label{eq:dav11}
\sum_{n=1}^\infty a_{n} \{nx\}=\sum_{m=1}^\infty c_{m} \sin(2\pi mx)
\quad\mbox{with}\quad
c_{m}=-\frac{1}{\pi m} \sum_{n\in\N \atop n|m} n a_{n}.
\end{equation}

One of the fascinating aspects of these expansions is that their study lies at the crossroad between several areas of mathematics. They appear naturally in several problems related with analytic number theory; this actually was the motivation of H.~Davenport for studying them, see also the recent studies by R.~de la Bret\`eche and G.~Tenenbaum (such as ~\cite{Breteche:2004uq} for instance). They were later considered in connexion with harmonic analysis, see {\em e.g.}~\cite{Jaffard:2004os} and references therein where a function space point of view is developed, and it is shown in which sense an arbitrary one-periodic odd function can be expanded on this system. Convergence properties of these series at particular points are related with the Diophantine approximation properties of these points, see~\cite{Breteche:2004uq,Jaffard:2004os}. Recently, J.~Br\'emont studied the $L^2$ and almost-sure convergence of these series, see \cite{Bremont:2011fk}. The multifractal analysis of these functions shows connexions between their pointwise regularity properties and geometric measure theory, see~\cite{Jaffard:2004os}, and also~\cite{Zhou:2009fk} for an extension of Davenport series with translated phases. Note also that examples of Davenport series valued in $\R^{2}$ were proposed by H.~Lebesgue as space filling functions; this study was developed in~\cite{Jaffard:2009ve,Jaffard:2010qf}, where the connexions between Davenport series and space filling functions are examined.

In this paper, we shall investigate the multivariate case, which has not been considered up to now. Our main motivation is that multivariate Davenport series are natural examples of multifractal fields. The recent increase of interest in such fields is motivated by the relevance of multifractal analysis techniques in image classification, see~\cite{AJW1,AJW2}. Indeed, the validation of 2D multifractal analysis algorithms requires the introduction and the mathematical study of collections of multifractal fields of various kinds. However, very few multivariate multifractal models have been studied up to now (see however the PhD thesis of H.~Oppenheim~\cite{Oppenheim:1997kx} for an early analysis of a multifractal function of several variables where Diophantine approximation properties are involved, and~\cite{Aubry:2002up} for fields generated by random wavelet series). Another case of random fields which have recently been studied are L\'evy fields, which are a natural extension of L\'evy processes to the multivariate setting; their multifractal analysis has recently been performed by the authors, see~\cite{Durand:2010uq}. The scarcity of existing results is partly due to the fact that the derivation of the multifractal properties of multivariate functions lies on variants of ubiquity methods which can prove much more involved in the multidimensional setting. Therefore, extending the collection of available multivariate models, and elucidating their multifractal properties is an important issue. An additional motivation of this paper is to draw a comparison between the multifractal behavior of Davenport series and L\'evy fields. Indeed, they are both constructed as superpositions of piecewise linear functions which display jumps along hyperplanes, the main difference being that the locations of these hyperplanes are random in the case of L\'evy fields, whereas they are determined by arithmetic conditions in the case of Davenport series. We shall see that the multifractal properties of Davenport series bear similarities with those of L\'evy fields, so that they can be seen as a kind of deterministic counterpart of these fields.

The paper is organized as follows. In Section~\ref{sec:relFourierDav}, we establish the relationships between the Davenport and Fourier coefficients of Davenport series, in the normally convergent case. We shall see that this relationship extends to more general, and actually distributional, settings in Section~\ref{sec:sobol}. The main purpose of this paper is the study of pointwise regularity properties of Davenport series. The key step consists in analyzing the locations and magnitudes of the jumps of Davenport series. Preliminary results concerning this study are collected in Sections~\ref{sec:disc} and~\ref{sec:jumpop}. In Section~\ref{sec:upbnd}, an upper bound of the H\"older exponent is derived and, as a consequence, cases where this exponent vanishes everywhere are worked out. A difficult question (which is far from being closed, even in the one-variable case) is to understand when this upper bound is sharp; the purpose of Section~\ref{sec:sparse} is to show that this is the case when the frequencies of the Davenport series are sufficiently sparse. Implications for multifractal analysis are stated in Section~\ref{sec:multifrac}. In Section~\ref{sec:sobol}, we shall consider convergence properties of Davenport series in the Sobolev spaces $H^{s}$ for $s\in\R$, especially when the sequence of coefficients does not belong to $\ell^{1}$; this study will also be the occasion to draw bridges with arithmetic functions in several variables, a topic which has been barely scratched until now (see however~\cite{Breteche:2001ys,Essou} and references therein). Concluding remarks and open problems are collected in Section~\ref{sec:conclud}. Finally, the proofs of the main results are completed in Sections~\ref{sec:proofthmformhold} and~\ref{sec:proofthmmultifrac}. This is the occasion for us to exhibit deep connexions with the theory of sets of large intersection and with the Duffin-Schaeffer and Catlin conjectures in the metric theory of Diophantine approximation, see Section~\ref{subsec:sliLaalpha}.

\section{Relationships between Davenport and Fourier series}\label{sec:relFourierDav}

We start by establishing some conventions. First, note that functions such as~(\ref{eq:dave}) are necessarily odd and $\Z^d$-periodic. Since $\{-x\}=-\{x\}$, it follows that the system supplied by the $\{n\cdot x\} $, for $n\in\Z^d$, is redundant. The choice made for one-dimensional Davenport series is to use only these functions for $n\geq 1$, as {\em e.g.}~in~(\ref{eq:dav11}) above. We shall make a different choice in dimension $d\geq 2$, which will preserve the symmetry of the decomposition. Specifically, we shall keep both functions $\{n\cdot x\}$ and $\{-n\cdot x\} $, and, without loss of generality, we shall assume that the sequence $(a_n)_{n\in\Z^d}$ of Davenport coefficients is an odd sequence indexed by $\Z^{d}$, which implies uniqueness of the decomposition.

The function spaces that we shall consider are composed of $\Z^{d}$-periodic odd functions, and the sequence spaces that we shall consider are composed of odd sequences. Therefore, we shall use the following conventions concerning spaces: With a slight abuse of notations, $\ell^{p}$ will denote the space of odd sequences which belong to $\ell^{p}(\Z^{d})$, $L^{2}$ is the space of odd locally square-integrable functions which are $\Z^d$-periodic, and, more generally, if $E$ is a space of functions defined on $\R^{d}$, we shall also denote by $E$ the space of odd functions that belong locally to $E$ and are $\Z^{d}$-periodic.

Another convention concerns divisibility in several dimensions. Let $n,m\in\ZD$. If $m=ln$ for some $l\in\Z^{\ast}$, we say that $l$ and $n$ are divisors of $m$. The fact that the term ``divisor'' applies without distinction to elements of $\ZD$ and of $\Z^{\ast}$ will not create confusions because the context will always be clear. If $l=\pm 1$ are the only integer divisors of $m$, we say that $m$ is {\em irreducible}; this means that its components are coprime. Throughout the paper, $\N$ denotes the set of positive integers.

Finally, the {\em support} of a sequence $a=(a_{n})_{n\in\Z^{d}}$ is 
\[
\supp{a}=\{n\in\Z^{d} \:|\: a_n\neq 0\}.
\]

Let us now investigate the relationship between Davenport and Fourier series. To this end, let us assume that the series~(\ref{eq:dave}) converges normally, {\em i.e.}~that the sequence $(a_{n})_{n\in\Z^{d}}$ belongs to $\ell^1$ (convergence properties in different functional settings will be investigated in Section~\ref{sec:sobol}). Then, $f$ belongs to $L^{\infty}$, hence to $L^{2}$ and the Fourier series expansion of $f$ converges in $L^{2}$. Since $f$ is odd, it may be written in the form
\[
f(x)=\sum_{m\in\Z^d} c_{m} \sin(2\pi m\cdot x) \quad\mbox{with}\quad c_{m}=\int_{[0,1)^{d}} f(x) \sin(2\pi m\cdot x)\,\dd x.
\]
Here, we adopt the same convention for Fourier series as for Davenport series, that is, we assume that the expansion is taken on all frequencies of $\Z^d$, but that the sequence $(c_{m})_{m\in\Z^d}$ is odd. Since~(\ref{eq:dave}) is normally convergent,
\[
c_{m}=\sum_{n\in\ZD} a_{n} \int_{[0,1)^{d}} \{n\cdot x\} \sin(2\pi m\cdot x) \,\dd x
\]
for all $m\in\ZD$. A straightforward computation shows that the above integral is equal to zero except if $n$ is a divisor of $m$, in which case there exists an integer $l\in\Z^{\ast}$ such that $ln=m$, and the integral is equal to $-1/(2\pi l)$. As a consequence, the Fourier coefficients of $f$ are given by
\begin{equation}\label{eq:davfou}
c_{m}=-\frac{1}{2\pi}\sum_{(l,n)\in\Z^{\ast}\times\ZD \atop ln=m}\frac{a_{n}}{l}.
\end{equation}

Note that without making any assumption on the summability of the sequence $(a_{n})_{n\in\Z^{d}}$, the above formula still enables us to define a sequence $(c_{m})_{m\in\Z^{d}}$. This detour via Fourier series will allow us to study the convergence of the series~(\ref{eq:dave}) even when $(a_{n})_{n\in\Z^{d}}$ does not belong to $\ell^{1}$, see Section~\ref{sec:sobol}. Indeed, we shall see that, in many functional settings, when the associated Fourier series converges, then the partial sums of the series $\sum_{n} a_{n}\{n\cdot x\}$ converge to the same limit.

\section{Discontinuities of Davenport series}\label{sec:disc}

Let us consider a bounded function $g:\R^{d}\to\R$. By definition, the magnitude of the jump of $g$ at any fixed point $x_{0}\in\R^{d}$ is
\[
\Delta_{g}(x_{0})=\limsup_{x\to x_{0}} g(x) - \liminf_{x\to x_{0}} g(x)
\]
(which may possibly vanish, in which case $g$ is continuous at $x_{0}$). The magnitude of the jumps can also be expressed by means of local oscillations. To be specific, let us recall that the {\em oscillation} of the function $g$ on a bounded subset $\Omega$ of $\R^d$ is defined by
\[
\osc_{g}(\Omega)=\sup_{x\in\Omega} g(x) - \inf_{x\in\Omega} g(x).
\]
Letting $\opball{x}{r}$ denote the open ball with center $x$ and radius $r$, it is easy to see that the magnitude of the jump of the function $g$ at the point $x_{0}$ satisfies
\[
\Delta_{g}(x_{0})=\lim_{r\to 0} \osc_{g}(\opball{x_{0}}{r}).
\]
We shall now determine the set of points at which the Davenport series $f$ defined by~(\ref{eq:dave}) has a discontinuity, and we shall study the magnitude of the corresponding jump. We assume in what follows that the sequence $(a_{n})_{n\in\Z^{d}}$ belongs to $\ell^{1}$.

Given a vector with integer coordinates $q\in\ZD$ and an integer $p\in\Z$, let $H_{p,q}$ denote the hyperplane
\begin{equation}\label{eq:defHpq}
H_{p,q}=\{x\in\R^{d} \:|\: p=q\cdot x\}.
\end{equation}
It is clear that multiplying $p$ and the components of $q$ by a common integer value leaves the hyperplane unchanged. In order
 to ensure the uniqueness of the representation, it is sufficient to assume that $p$ and the components of $q$ are coprime and that $q$ belongs to the subset $\Z^{d}_{+}$ of $\ZD$ formed by the vectors whose first nonvanishing coordinate is positive. In fact, one easily checks that any hyperplane $H_{p,q}$ may be indexed in a unique manner by a pair $(p,q)$ that belongs to
\[
\Hcal_{d}=\left\{(p,q)\in\Z\times\Z^{d}_{+} \:|\: \gcd(p,q)=1\right\},
\]
where $\gcd(p,q)$ is the greatest common divisor of the integer $p$ and the components of the vector $q$. Furthermore, let $\{\,\cdot\,\}_{\star}$ denote the restriction of the sawtooth function $\{\,\cdot\,\}$ to the open interval $(-1/2,1/2)$. Then, $\{\,\cdot\,\}_{\star}$ is continuous everywhere except at the origin: $\{x\}_{\star}$ makes a jump of size $-1$ when $x$ crosses zero in the upward direction. In addition, $\{x\}$ is the sum of $\{x-p\}_{\star}$ over all the integers $p\in\Z$. Along with the fact that the sequence $(a_{n})_{n\in\Z^{d}}$ and the function $\{\,\cdot\,\}_{\star}$ are both odd, this enables us to rewrite the definition~(\ref{eq:dave}) of the Davenport series $f$ in the form
\[
f(x)=\sum_{(p,q)\in\Hcal_{d}} f_{p,q}(x)
\qquad\mbox{with}\qquad
f_{p,q}(x)=2\sum_{l=1}^{\infty} a_{lq}\{l(q\cdot x-p)\}_{\star},
\]
where all the series converge normally. This decomposition enlightens the fact that $f$ is the superposition of a family of functions that are continuous everywhere except maybe on a specific hyperplane of the above kind. To be precise, each function $f_{p,q}$ is continuous everywhere except maybe on the hyperplane $H_{p,q}$, and the fact that $(p,q)$ belongs to $\Hcal_{d}$ implies that these hyperplanes are distinct. Moreover, when a point $x$ crosses $H_{p,q}$, the real points $l(q\cdot x-p)$, for $l\geq 1$, all cross zero, so that $f_{p,q}(x)$ makes a jump of magnitude $|A_{q}|$, where
\begin{equation}\label{eq:defAq}
A_{q}=2\sum_{l=1}^{\infty} a_{lq}.
\end{equation}
Note that, in the case where the latter sum vanishes, the function $f_{p,q}$ is actually continuous on the whole space, including the hyperplane $H_{p,q}$.

The analysis of the discontinuities of the Davenport series $f$ begins with a first remark: As the series~(\ref{eq:dave}) is normally convergent, its sum $f$ is a function in $L^{\infty}$, so that the potential discontinuities must have finite magnitude. We shall now show that the set of points at which $f$ is not continuous is exactly
\begin{equation}\label{eq:defdisc}
\bigcup_{(p,q)\in\Hcal_{d} \atop A_{q}\neq 0} H_{p,q}.
\end{equation}
First, note that, if a point $x_{0}$ does not belong to the latter set, then the above decomposition entails that $f$ is a sum of uniformly convergent series of functions that are continuous at $x_{0}$, thereby being continuous at $x_{0}$ as well. Conversely, if a point $x_{0}$ belongs to a hyperplane $H_{p,q}$ indexed by a pair $(p,q)\in\Hcal_{d}$ for which $A_{q}$ does not vanish, and to no other hyperplane of that form (which is the case of Lebesgue-almost every point of $H_{p,q}$), then $f$ has a discontinuity at $x_{0}$ of magnitude $|A_{q}|$, that is,
\[
\Delta_{f}(x_{0})=|A_{q}|>0.
\]
More generally, suppose that $x_0$ belongs to a (possibly infinite) collection of hyperplanes $H_{p_{i},q_{i}}$ indexed by pairs $(p_{i},q_{i})\in\Hcal_{d}$ for which $A_{q_{i}}$ do not vanish. The previous case shows that, for any specific value of $i$, the function $f$ has a discontinuity of magnitude exactly $|A_{q_{i}}|$ on a dense set of points of $H_{p_{i},q_{i}}$. Therefore, one can pick a point $y_{i}$ arbitrarily close to $x_{0}$ at which $f$ has a discontinuity of magnitude exactly $|A_{q_{i}}|$. It follows that
\[
\Delta_{f}(x_{0})\geq\max_{i}|A_{q_{i}}|>0,
\]
so that $f$ exhibits a discontinuity at $x_{0}$.

Given that $|A_{q}|$ is the magnitude of the jump of the Davenport series $f$ at Lebesgue-almost every point of the hyperplane $H_{p,q}$, we shall call with a slight abuse $|A_{q}|$ the magnitude of the jump of $f$ on $H_{p,q}$.

We see here a sharp contrast with Fourier series: The series~(\ref{eq:dave}) will usually exhibit discontinuities no matter how fast the coefficients $a_{n}$ decay. The following proposition shows that even more is true: The zero function is the only continuous Davenport series.

\begin{prp}
Let $f$ be a Davenport series with coefficients given by a sequence $a=(a_{n})_{n\in\Z^{d}}$ in $\ell^{1}$. If $f$ is a continuous function, then
\[
\forall n\in\Z^{d} \qquad a_{n}=0.
\]
\end{prp}

\begin{proof}
The continuity of $f$ implies that $A_{q}=0$ for all vectors $q\in\Z^{d}$. Given an irreducible vector $q$, let $b^{q}_{l}=a_{lq}$ for any integer $l\geq 1$. Then, the sequence $(b^{q}_{l})_{l\geq 1}$ is in $\ell^{1}(\N)$ and satisfies
\[
\forall l\geq 1 \qquad \sum_{k=1}^{\infty} b^{q}_{kl} =0.
\]
Haar proved that these conditions imply that $b^{q}_{l}=0$ for all $l\geq 1$, see~\cite[Chapter~1, no.~129]{Polya:1972yq}. This argument holds in all directions $q$, so the result follows.
\end{proof}

We refer to the next section for more general results that explain how to recover the coefficients $a_{n}$ from the values $A_{q}$.

\section{The jump operator}\label{sec:jumpop} 

In order to study the regularity properties of normally convergent Davenport series, it is useful to consider the linear operator $J$ which maps the sequence of Davenport coefficients to the sequence of jumps, and which is defined by
\begin{equation}\label{eq:defJ}
\forall (a_{n})_{n\in\Z^{d}}\in\ell^{1} \qquad J\left((a_{n})_{n\in\Z^{d}}\right)=(A_{q})_{q\in\Z^{d}}\in\ell^{\infty},
\end{equation}
where the coefficients $A_{q}$ are given by~(\ref{eq:defAq}). The key results concerning this mapping follow from those obtained in~\cite{Jaffard:2004os} in the one-dimensional case; this is due to a remarkable decomposition that we now present.

Let $\irrd$ denote the subset of $\ZD$ formed by the irreducible vectors and let $V_{m}$ denote the vector space of odd sequences $(a_{n})_{n\in\Z^{d}}$ that are supported by the multiples of such a vector $m\in\irrd$, {\em i.e.}~such that
\[
\supp{(a_{n})_{n\in\Z^{d}}}\subseteq \Z\,m. 
\]
Any odd sequence indexed by $\Z^{d}$ may be decomposed as a sum of sequences $b^{m}$ such that $b^{m}\in V_{m}$. However, as the vector subspaces $V_{m}$ and $V_{-m}$ coincide, in order to ensure the uniqueness of the decomposition, we shall privilege the irreducible vectors whose first nonvanishing coordinate is positive. The set of those vectors is therefore $\irrd_{+}=\irrd\cap\Z^{d}_{+}$. As a consequence, we obtain the following unconditional Schauder decomposition:
\begin{equation}\label{eq:ell1oplus}
\ell^{1}=\bigoplus_{m\in\irrd_{+}}\left(V_{m}\cap\ell^{1}\right),
\end{equation}
meaning that any sequence in $\ell^{1}$ may be written in a unique manner as the sum in the $\ell^{1}$ sense of an unconditionally summable family indexed by $m\in\irrd_{+}$ of sequences in $V_{m}\cap\ell^{1}$. Moreover, the operator $J$ maps the subspace $V_{m}\cap\ell^{1}$ to $V_{m}\cap\ell^{\infty}$. It follows that, in order to study $J$, it suffices to analyze its restriction $J_{m}$ to the subspace $V_{m}\cap\ell^{1}$, for any fixed vector $m\in\irrd_{+}$.

On top of that, let $S_{m}$ denote the operator of subsampling with step $m$, which is defined by
\[
S_{m}((a_{n})_{n\in\Z^d})=(a_{lm})_{l\geq 1}
\]
for any odd sequence $(a_{n})_{n\in\Z^d}$. As we consider odd sequences only, it is clear that the restriction of $S_{m}$ to $V_{m}$ is one-to-one. In fact, we even see that $S_{m}$ maps $V_{m}\cap\ell^{p}$ onto $\ell^{p}(\N)$.

In the one-dimensional case, as already mentioned above and illustrated by~(\ref{eq:dav11}), one assumes that the sequence of Davenport coefficients is supported on $\N$, instead of supposing that they form an odd sequence indexed by $\Z$. Thus, the operator $J$ has a simpler counterpart which has already been considered in~\cite{Jaffard:2004os}; this is the {\em jump operator} $\Jcal$ defined by
\begin{equation}\label{eq:defJcal}
\Jcal\left((b_{n})_{n\geq 1}\right)=\left(\sum_{l=1}^{\infty}b_{lq}\right)_{q\geq 1}\in\ell^{\infty}(\N),
\end{equation}
for any sequence $(b_{n})_{n\geq 1}\in\ell^{1}(\N)$. The following straightforward lemma shows that all the mappings $J_{m}$ essentially reduce to $\Jcal$.

\begin{lem}\label{lem:SmJm}
Let us consider a vector $m\in\irrd_{+}$. Then, for any sequence $a=(a_{n})_{n\in\Z^{d}}$ in $V_{m}\cap\ell^{1}$, 
\[
S_{m}(J_{m}(a))=2\Jcal(S_{m}(a)).
\]
Therefore, the following diagram is commutative:
\[
\xymatrix{
    V_{m}\cap\ell^{1} \ar[r]^{J_{m}} \ar[d]^{S_{m}}_{\sim}  & V_{m}\cap\ell^{\infty} \ar[d]^{S_{m}}_{\sim} \\
    \ell^{1}(\N) \ar[r]^{2\Jcal} & \ell^{\infty}(\N)
}
\]
\end{lem}

It follows from Lemma~\ref{lem:SmJm} that the mapping $J$ can be inverted on each subspace $V_{m}\cap\ell^{1}$ by means of an inversion formula for the one-dimensional jump operator $\Jcal$. This formula has been obtained in~\cite{Jaffard:2004os} and is recalled in the statement of Proposition~\ref{prp:invJcalTau} below. It makes use of the {\em M\"obius function} $\mu$, which is defined on the positive integers by $\mu(n)=0$ if $n$ is not square-free, and by $\mu(n)=(-1)^{k}$ if $n$ is square-free and admits exactly $k$ prime divisors. The inversion formula holds on the subspace $\Tau(\N)$ of $\ell^{1}(\N)$ that is formed by the sequences $(b_{n})_{n\geq 1}$ for which the series $\sum_{n}\tau(n)|b_{n}|$ converges, where $\tau(n)$ denotes the number of divisors of $n$; this merely means that the restriction of $\Jcal$ to that subspace is one-to-one. It is well-known that the sequence $\tau(n)$ grows slower than any positive power of $n$, in the sense that $\tau(n)={\rm o}(n^{\eps})$ as $n$ goes to infinity, for all $\eps>0$. This is a plain consequence of the fact that
\begin{equation}\label{eq:asymptau}
\limsup_{n\to\infty}\frac{\log\log n}{\log n}\log\tau(n)=\log 2,
\end{equation}
see {\em e.g.}~\cite[Theorem~13.12]{Apostol:1976uq}. This implies in particular that, for any real $\gamma$ larger than one, $\Tau(\N)$ contains the space $\Fcal^{\gamma}(\N)$ of all the sequences $b=(b_{n})_{n\geq 1}$ such that
\[
|b|_{\Fcal^{\gamma}(\N)}=\sup_{n\geq 1} n^{\gamma}|b_{n}|<\infty.
\]
Thus, the restriction of $\Jcal$ to each $\Fcal^{\gamma}(\N)$ is one-to-one. The next result even shows that $\Jcal$ is a bicontinuous automorphism of $\Fcal^{\gamma}(\N)$. We refer to~\cite{Jaffard:2004os} for its proof.

\begin{prp}\label{prp:invJcalTau}
The operator $\Jcal$ induces a one-to-one mapping from $\Tau(\N)$ into $\ell^{1}(\N)$. More specifically, for any sequence $B=(B_{q})_{q\geq 1}$ in the image set $\Jcal(\Tau(\N))$, the equation
\[
B=\Jcal(b)
\]
admits exactly one solution $b=(b_{n})_{n\geq 1}$ in $\Tau(\N)$, namely, the sequence defined by
\begin{equation}\label{eq:invJcal}
\forall n\geq 1 \qquad b_{n}=\sum_{l=1}^\infty \mu(l) B_{ln}.
\end{equation}
Moreover, for any real $\gamma>1$, the operator $\Jcal$ induces a bicontinuous automorphism of the space $\Fcal^{\gamma}(\N)$ whose inverse is given by~(\ref{eq:invJcal}).
\end{prp}

Thanks to Lemma~\ref{lem:SmJm}, Proposition~\ref{prp:invJcalTau} naturally extends to the multivariate setting. In fact, we now establish that the higher-dimensional jump operator $J$ induces a bicontinuous automorphism on the space $\Fcal^{\gamma}$ defined as follows.

\begin{df}
The space $\Fcal^{\gamma}$ is the vector space composed of the odd sequences $a=(a_{n})_{n\in\Z^{d}}$ satisfying
\[
|a|_{\Fcal^{\gamma}}=\sup_{n\in\ZD} |n|^{\gamma}|a_{n}|<\infty.
\]
\end{df} 

If $\gamma$ is larger than the dimension $d$ of the ambient space, it is clear that $\Fcal^{\gamma}$ may be seen as a vector subspace of $\ell^{1}$, so that the operator $J$ is well-defined on $\Fcal^{\gamma}$. In the opposite case, $\Fcal^{\gamma}$ is not necessarily included in $\ell^{1}$. However, if $\gamma>1$, the formula~(\ref{eq:defJ}) still has a meaning, because all the series~(\ref{eq:defAq}) converge, which enables us to define the operator $J$ on $\Fcal^{\gamma}$ as well. This may also be seen as a consequence of Lemma~\ref{lem:SmJm}, along with the fact that $S_{m}(\Fcal^{\gamma})$ is contained in $\ell^{1}(\N)$.

\begin{prp}\label{prp:equivdec}
For any $\gamma>1$, the jump operator $J$ is a bicontinuous automorphism of $\Fcal^{\gamma}$, and its inverse is given by
\begin{equation}\label{eq:invJ}
\forall A=(A_{q})_{q\in\Z^{d}}\in\Fcal^{\gamma} \qquad J^{-1}(A)=\left(\frac{1}{2}\sum_{l=1}^{\infty}\mu(l)A_{ln}\right)_{n\in\Z^{d}}.
\end{equation}
\end{prp}

\begin{proof}
Given $\gamma>1$, let $a=(a_{n})_{n\in\Z^{d}}$ be a sequence in $\Fcal^{\gamma}$, and let $A=(A_{q})_{q\in\Z^{d}}$ denote its image under $J$, that is, $A=J(a)$. Then, for each vector $q\in\ZD$,
\[
|A_{q}|\leq 2\sum_{l=1}^{\infty} |a_{lq}|\leq 2\sum_{l=1}^{\infty}\frac{|a|_{\Fcal^{\gamma}}}{|lq|^{\gamma}} \leq\frac{2\zeta(\gamma)}{|q|^{\gamma}}|a|_{\Fcal^{\gamma}},
\]
where $\zeta$ is the Riemann zeta function. Therefore, $J$ is continuous on $\Fcal^{\gamma}$ with operator norm at most $2\zeta(\gamma)$.

In order to study the invertibility of the operator $J$ on $\Fcal^{\gamma}$, let us begin by observing that, in a way similar to~(\ref{eq:ell1oplus}), the latter space may be decomposed as the direct sum over $m\in\irrd_{+}$ of the subspaces $V_{m}\cap \Fcal^{\gamma}$. Moreover, the subsampling operator $S_{m}$ induces a one-to-one mapping from $V_{m}\cap \Fcal^{\gamma}$, which is included in $V_{m}\cap\ell^{1}$, onto $\Fcal^{\gamma}(\N)$. Then, letting $\pi_{m}$ denote the projection onto $V_{m}$, we deduce from Lemma~\ref{lem:SmJm} and Proposition~\ref{prp:invJcalTau} that the following diagram is commutative:
\[
\xymatrix{
    \Fcal^{\gamma} \ar[r]^{J} \ar@{->>}[d]^{\pi_{m}} & \Fcal^{\gamma} \ar@{->>}[d]^{\pi_{m}}\\
    V_{m}\cap \Fcal^{\gamma} \ar[r]^{J_{m}} \ar[d]^{S_{m}}_{\sim} & V_{m}\cap \Fcal^{\gamma} \ar[d]^{S_{m}}_{\sim} \\
    \Fcal^{\gamma}(\N) \ar[r]^{2\Jcal}_{\sim} & \Fcal^{\gamma}(\N)
}
\]
This implies the invertibility of the operator $J$ on the space $\Fcal^{\gamma}$. In order to obtain an explicit formula for its inverse $J^{-1}$, let us make use of the diagram to infer that the equation $A=J(a)$ implies
\[
S_{m}(\pi_{m}(a))=\frac{1}{2}\Jcal^{-1}(S_{m}(\pi_{m}(A)))
\]
for all $m\in\irrd_{+}$, where $\Jcal^{-1}$ denotes the inverse of the one-dimensional jump operator $\Jcal$ on $\Fcal^{\gamma}(\N)$. By means of~(\ref{eq:invJcal}), we deduce that
\[
S_{m}(a)=\left(\frac{1}{2}\sum_{l=1}^{\infty}\mu(l) A_{lkm}\right)_{k\geq 1}
\]
and~(\ref{eq:invJ}) follows. Finally, the method that we used above in order to show the continuity of the operator $J$ also applies to its inverse because the M\"obius function is at most one in absolute value; proceeding in this way, we deduce that $J^{-1}$ is continuous with operator norm at most $\zeta(\gamma)/2$.
\end{proof}

We shall show in Section~\ref{sec:sparse} below that the statement of Proposition~\ref{prp:equivdec} may be extended to the case where $0<\gamma\leq 1$, up to replacing $\Fcal^{\gamma}$ by a subspace formed by sequences whose support satisfies an additional sparsity assumption.

\section{Pointwise H\"older regularity}\label{sec:upbnd}

The section contains general results on the pointwise regularity of multivariate functions with a dense set of discontinuities, a class in which the typical Davenport series fall. We begin by recalling the appropriate definitions.

\begin{df}\label{df:hol}
Let $f:\R^{d}\to\R$ be a locally bounded function, $x_{0}\in\R^{d}$ and $\alpha\geq 0$. The function $f$ belongs to $C^{\alpha}(x_{0})$ if there exist $C>0$ and a polynomial $P_{x_{0}}$ of degree less than $\alpha$ such that for all $x$ in a neighborhood of $x_{0}$,
\begin{equation}\label{eq:df:Calphax0}
|f(x)-P_{x_{0}}(x)|\leq C\,|x-x_{0}|^{\alpha}.
\end{equation}
The {\em H\"older exponent} of $f$ at $x_{0}$ is then defined by
\[
h_{f}(x_{0})=\sup\{\alpha\geq 0 \:|\: f\in C^{\alpha}(x_{0})\}.
\]
\end{df}

Note that $h_{f}$ takes values in $[0,\infty]$. The following lemma yields an upper bound on the pointwise H\"older exponent of functions that have a dense set of discontinuities, and will be applied to Davenport series in the following. It is a direct extension to the multivariate setting of Lemma~1 in~\cite{Jaffard:1997jx}.

\begin{lem}\label{lem:upbndjump}
Let $f:\R^{d}\to\R$ be a locally bounded function and let $x_{0}\in\R^{d}$. Then,
\begin{equation}\label{eq:lem:upbndjump}
h_{f}(x_{0})\leq\liminf_{s\to x_{0}} \frac{\log\Delta_{f}(s)}{\log|s-x_{0}|}.
\end{equation}
\end{lem}

\begin{proof}
We may obviously assume that $h_{f}(x_{0})$ is positive, in which case $f$ is continuous at $x_{0}$. Then, given a positive real $\alpha$ less than $h_{f}(x_{0})$, there exists $\delta,C>0$ and a polynomial $P_{x_{0}}$ such that~(\ref{eq:df:Calphax0}) holds for any $x$ in the open ball $\opball{x_{0}}{\delta}$.

Now, let $s$ be a discontinuity point of $f$, which thus necessarily differs from $x_{0}$. The magnitude $\Delta_{f}(s)$ of the jump of $f$ at $s$ is positive, as well as $\eps=\Delta_{f}(s)/6$. Owing to the definition of $\Delta_{f} (s)$ and the continuity of the polynomial $P_{x_{0}}$, there exist two points $x_{1}$ and $x_{2}$ such that
\[
|f(x_{1})-f(x_{2})|\geq\Delta_{f}(s)-\eps \qquad\mbox{and}\qquad |P_{x_{0}}(x_{1})-P_{x_{0}}(x_{2})|\leq\eps.
\]
These two points may be chosen arbitrarily close to $s$, for instance within range $|s-x_{0}|/2$ from that point. Therefore, at least one of them, denoted by $x(s)$, satisfies
\begin{equation}\label{eq:condxs}
|x(s)-s|\leq\frac{|s-x_{0}|}{2} \qquad\mbox{and}\qquad |f(x(s))-P_{x_{0}}(x(s))|\geq\frac{\Delta_{f}(s)}{3}.
\end{equation}

Let $L$ denote the right-hand side of~(\ref{eq:lem:upbndjump}), which we may obviously assume to be finite. Thus, there exists a sequence $(s_{n})_{n\geq 1}$ of discontinuity points of $f$ which realizes the lower limit $L$; the points $s_{n}$ necessarily differ from $x_{0}$ but they converge to that point. For each integer $n\geq 1$, the above procedure yields a point $x(s_{n})$ for which~(\ref{eq:condxs}) holds. The resulting sequence $(x(s_{n}))_{n\geq 1}$ thus satisfies
\[
\limsup_{n\to\infty}\frac{\log|f(x(s_{n}))-P_{x_{0}}(x(s_{n}))|}{\log|x(s_n)-x_0|}\leq L.
\]
In the meantime,~(\ref{eq:df:Calphax0}) implies that
\[
\liminf_{n\to\infty}\frac{\log|f(x(s_{n}))-P_{x_{0}}(x(s_{n}))|}{\log|x(s_n)-x_0|}\geq\alpha.
\]
We deduce that $\alpha\leq L$, and the result follows from the fact that $\alpha$ may be chosen arbitrarily close to the H\"older exponent $h_{f}(x_{0})$.
\end{proof}

Let us apply Lemma~\ref{lem:upbndjump} to multivariate Davenport series. For each vector $q$ in $\ZD$, let $\prim{q}$ denote the set of all integers $p\in\Z$ such that $\gcd(p,q)=1$. We remark that a vector $q$ is irreducible, {\em i.e.}~belongs to $\irrd$, if and only if $\prim{q}$ contains zero. Moreover, the set $\Hcal_{d}$ which indexes the hyperplanes $H_{p,q}$ in a unique manner is actually formed by the pairs $(p,q)$ such that $p\in\prim{q}$ and $q\in\Z^{d}_{+}$. Now, given a point $x_{0}\in\R^{d}$, let $\delta^{\Pcal}_{q}(x_{0})$ denote the distance between $x_{0}$ and the hyperplanes $H_{p,q}$ with $p\in\prim{q}$, that is,
\begin{equation}\label{eq:defdPqx0}
\delta^{\Pcal}_{q}(x_{0})=\dist\left(x_{0},\bigcup_{p\in\prim{q}} H_{p,q}\right)=\frac{1}{|q|}\inf_{p\in\prim{q}} |q\cdot x_{0}-p|,
\end{equation}
where $|q|$ is the Euclidean norm of the vector $q$. It is obvious that $\delta^{\Pcal}_{-q}(x_{0})$ coincides with $\delta^{\Pcal}_{q}(x_{0})$, so there is no loss of information in assuming, whenever necessary, that $q$ is in $\Z^{d}_{+}$. Also, it is easy to see that the set $\prim{q}$ is invariant under the translations of the form $p\mapsto p+k\cdot q$ with $k\in\Z^{d}$, which makes it clear that the function $\delta^{\Pcal}_{q}$ is $\Z^{d}$-periodic.

The analysis of the discontinuities of Davenport series that we led in Section~\ref{sec:disc} above ensures that every hyperplane $H_{p,q}$ indexed by $(p,q)\in\Hcal_{d}$ contains a dense set of points at which the Davenport series has a discontinuity of magnitude $|A_{q}|$, with the proviso that the sum $A_{q}$ defined by~(\ref{eq:defAq}) does not vanish. These observations then yield the following corollary to Lemma~\ref{lem:upbndjump}.
 
\begin{cor}\label{cor:upbndjumpDav}
Let $f$ be a Davenport series with $a=(a_{n})_{n\in\Z^{d}}\in\ell^{1}$. Then,
\[
\forall x_{0}\in\R^{d} \qquad h_{f}(x_{0})\leq\liminf_{q\to\infty\atop q\in\supp{A}}\frac{\log|A_{q}|}{\log \delta^{\Pcal}_{q}(x_{0})},
\]
where the sequence $A=(A_{q})_{q\in\Z^{d}}$ of jump sizes is the image of the sequence $a$ under the jump operator $J$ defined by~(\ref{eq:defJ}).
\end{cor}

In the above statement, we adopt the usual convention according to which the lower limit is infinite if the index set $\supp{A}$ is finite, in which case the bound is trivial. We now illustrate Corollary~\ref{cor:upbndjumpDav} by pointing out a class of Davenport series whose H\"older exponent vanishes everywhere, as a direct consequence of the previous upper bound. These series are characterized by the fact that the magnitude $|A_{q}|$ of the jumps does not become too small as $q$ goes to infinity along a subsequence satisfying particular arithmetical properties.

In order to specify these properties, let us begin by observing that $\delta^{\Pcal}_{q}(x_{0})$ may sometimes be bounded above by $1/|q|$ infinitely often, up to a logarithmic factor; also, in view of the statement of Corollary~\ref{cor:upbndjumpDav}, we may restrict our attention to the vectors $q$ for which $A_{q}$ does not vanish. The situation described above then occurs precisely when the support of the sequence $A$ of jump sizes is regular in the sense of the next definition, which makes use of the function $\kappa$ defined as follows: For every point $x_{0}\in\R^{d}$ and every infinite subset $Q$ of $\Z^{d}$,
\[
\kappa(x_{0},Q)=\limsup_{q\to\infty \atop q\in Q}\frac{\log\left( \inf\limits_{p\in\prim{q}} |q\cdot x_{0}-p| \right)}{\log|q|}.
\]
Note that, as a consequence of the periodicity of the function $\delta^{\Pcal}_{q}$, the function $\kappa(\,\cdot\,,Q)$ is $\Z^{d}$-periodic.

\begin{df}
\begin{enumerate}
\item An infinite subset $Q$ of $\Z^{d}$ is said to be regular if the following condition holds:
\[
\forall x_{0}\in\R^{d} \qquad \kappa(x_{0},Q)<1.
\]
\item A Davenport series with coefficients given by a sequence $a\in\ell^{1}$ is regular if $\supp{J(a)}$ is a regular subset of $\Z^{d}$.
\end{enumerate}
\end{df}

It is clear that any infinite subset of a regular set is also regular. Moreover, the fact that a set $Q$ is regular roughly means that the sets $\prim{q}$, for $q\in Q$, do not have exceptionally long gaps. In order to elaborate on this remark, let us focus on the one-dimensional case and give some heuristic arguments. In that situation, the sets $\prim{q}$ are $q\Z$-periodic, so that it suffices to analyze their gaps in the interval $\{1,\ldots,q-1\}$, which clearly amounts to examining the difference between two consecutive numbers prime to $q$. There are $\phi(q)$ such numbers, where $\phi$ denotes Euler's totient function. Hence, in the absence of exceptionally long gaps, the intervals between two consecutive numbers prime to $q$ would have length of the order of $q/\phi(q)$. Thus, the infimum arising in the definition of $\kappa(x_{0},Q)$ would grow at a comparable rate, up to constants, and $\kappa(x_{0},Q)$ would actually vanish. This is due to the fact that $q/\phi(q)={\rm O}(\log\log q)$ as $q\to\infty$; indeed, it is known that
\[
\liminf_{q\to\infty}\frac{\phi(q)}{q}\log\log q=\ee^{-\gamma},
\]
where $\gamma$ denotes the Euler-Mascheroni constant, see {\em e.g.}~\cite[Theorem~13.14]{Apostol:1976uq}. In general, though, there may exist exceptional gaps of length much larger or smaller than $q/\phi(q)$ between the numbers prime to $q$, so the above arguments are not always applicable. The literature seems rather scarce on that difficult topic, apart from a series of papers by C.~Hooley~\cite{Hooley:1962uq,Hooley:1965kx,Hooley:1965fj}.

In some cases, the previous heuristic arguments can be turned into a rigorous proof. For instance, let us suppose that $Q$ is the set $l^{\N}=\{l^{k},\ k\geq 1\}$ of integer powers of a given prime number $l\geq 2$. For any $k\geq 1$, the set $\prim{l^{k}}$ is formed by the nonmultiples of $l$, thereby having gaps of length one only. Hence, the infimum of $|l^{k}x_{0}-p|$ over $p\in\prim{l^{k}}$ is at most two, so that $\kappa(x_{0},l^{\N})=0$. A one-dimensional Davenport series whose coefficients $a=(a_{n})_{n\geq 1}$ are supported in $l^{\N}$ will be termed as $l$-adic. The sequence $A=J(a)$ of jump sizes is then also supported in $l^{\N}$. As a result, any $l$-adic Davenport series is regular.

Note however that in slightly more complicated examples, the above infimum may not easily be bounded by a constant, because the sets $\prim{q}$ may be chosen in such a way that they have longer gaps than above. As an illustration, still in dimension one, assume that $Q$ is the set of all primorials of prime numbers, that is, the set of all integers $q_{k}=p_{1}\cdots p_{k}$ for $k\geq 1$, where $p_{i}$ is the $i$-th prime number. Then, $\prim{q_{k}}$ is the sequence of integers prime that are not a multiple of any of the primes $p_{1},\ldots,p_{k}$. It follows that $\prim{q_{k}}$ has gaps of size at least $p_{k+1}-1$, which tends to infinity as $k\to\infty$.

We now introduce a definition that bears on the asymptotic behavior of the sequence indexed by $\Z^{d}$; we shall apply it in what follows to the sequence formed by the jump magnitudes $|A_{q}|$.

\begin{df}\label{df:slowdecay}
Let $b=(b_{q})_{q\in\Z^{d}}$ be a real-valued sequence and let $Q$ be an infinite subset of $\Z^{d}$. We say that the sequence $b$ has slow decay on $Q$ if
\[
\liminf_{q\to\infty \atop q\in Q}\frac{-\log|b_{q}|}{\log|q|}=0.
\]
\end{df}

Note that this condition is more and more restrictive as the subset $Q$ becomes smaller; to be precise, any sequence with slow decay on a given set has slow decay on all its supersets. In addition, the above definition is in stark contradiction with that of the sets $\Fcal^{\gamma}$ that we introduced in Section~\ref{sec:jumpop} and on which the jump operator $J$ is a bicontinuous automorphism. In fact, it is easy to see that the sequences belonging to the sets $\Fcal^{\gamma}$, for $\gamma>0$, do not have slow decay on $\Z^{d}$.

Combining the previous two definitions, and calling upon Corollary~\ref{cor:upbndjumpDav}, we readily deduce the following result.

\begin{cor}
Let $(a_{n})_{n\in\Z^{d}}$ be a sequence in $\ell^{1}$, and let $A=(A_{q})_{q\in\Z^{d}}$ be its image under the jump operator $J$. Let us assume that $A$ has slow decay on a regular subset of $\Z^{d}$. Then,
\[
\forall x_0\in\R^{d} \qquad h_{f}(x_{0})=0.
\]
\end{cor} 

A typical situation encompassed by the above setting is that of a regular Davenport series for which the sequence $A$ of jump magnitudes has slow decay. In dimension one, this is the case of the Davenport series of the form
\[
f_{l,\alpha}(x)=\sum_{k=1}^{\infty} \frac{\{l^k x\}}{k^\alpha}
\]
where $l$ is a prime number and $\alpha$ is larger than one. Indeed, $f_{l,\alpha}$ being an $l$-adic Davenport series, it is regular. Furthermore, let $b=(b_{n})_{n\geq 1}$ denote the sequence of its coefficients, namely, $b_{n}$ is equal to $k_{0}^{-\alpha}$ if $n=l^{k_{0}}$ for some integer $k_{0}\geq 1$, and vanishes otherwise. Then, one easily checks that the sequence of jump sizes $B=\Jcal(b)$ admits the same expression except that we have to replace $k_{0}^{-\alpha}$ by the sum $\sum_{k\geq k_{0}}k^{-\alpha}$ in the first case. It follows that $B$ has slow decay, and the previous result ensures that $h_{f_{l,\alpha}}(x_{0})$ vanishes everywhere.

\section{Sparse Davenport series}\label{sec:sparse} 

Recall that Corollary~\ref{cor:upbndjumpDav} above provides an upper bound on the H\"older exponent of a general Davenport series. In this section, we shall show that, under further assumptions, this bound gives the correct value of the H\"older exponent, which will ultimately enable us to perform the multifractal analysis of the corresponding series, see Section~\ref{sec:multifrac}. Our main assumption implies that there cannot be too many nonvanishing terms in the Davenport expansions, and boils down to a sparsity condition on the support of the sequence of Davenport coefficients. Note that, in one variable, the only case where one can determine the spectrum of singularities of Davenport series without additional assumptions on the coefficients is precisely the case where one assumes that the frequencies satisfy a lacunarity assumption, see~\cite{Jaffard:2010qf}. 

\subsection{Sparse sets and link with lacunary and Hadamard sequences}

Formally, we define the notion of sparse set in the following manner. Recall that $\opball{0}{R}$ denotes the open ball of $\R^{d}$ with center zero and radius $R$. In addition, we let $\card$ stand for cardinality.

\begin{df}
Let $Q$ be a nonempty subset of $\R^{d}$. The set $Q$ is said to be sparse if
\[
\lim_{R\to\infty}\frac{\log\card\left(Q\cap\opball{0}{R}\right)}{\log R}=0.
\]
\end{df}

In what follows, we say that an $\R^{d}$-valued sequence $\lambda=(\lambda_{n})_{n\geq 1}$ is sparse if the set of its values, namely, $\{\lambda_{n},\,n\geq 1\}$ is sparse in the sense of the above definition. If there is no redundancy in the sequence, {\em i.e.}~if it is injective, then the sparsity condition suggests that its terms do not accumulate excessively, but rather escape to infinity quite fast. Indeed, if $\lambda$ is sparse and injective, it is possible to rearrange its terms so as to assume that the sequence $(|\lambda_{n}|)_{n\geq 1}$ is nondecreasing. Then, one easily checks that the latter sequence grows faster than any power function at infinity, specifically,
\[
\lim_{n\to\infty}\frac{\log|\lambda_{n}|}{\log n}=\infty.
\]

A notable case where the sparsity condition holds is given by the sequences $(\lambda_{n})_{n\geq 1}$ that are both {\em separated}, meaning that
\[
\exists C>0 \quad \forall n,m\geq 1 \qquad n\neq m \quad\Longrightarrow\quad |\lambda_{n}-\lambda_{m}|\geq C,
\]
and {\em lacunary}, in the sense that
\[
\exists C'>0 \quad \forall n,m\geq 1 \qquad n\neq m \quad\Longrightarrow\quad |\lambda_{n}-\lambda_{m}|\geq C'(|\lambda_{n}|+|\lambda_{m}|);
\]
these two notions are standard in the study of nonharmonic Fourier series, see for instance~\cite{Jaffard:2010ys,Young:1980gf}. Indeed, let us assume that the two above conditions hold, and let $\Ncal_{j}(\lambda)$ collect the indices of the terms of the sequence within distance between $2^{j-1}$ and $2^{j}$ from the origin, that is,
\[
\Ncal_{j}(\lambda)=\{n\geq 1 \:|\: 2^{j-1}\leq|\lambda_{n}|<2^{j}\},
\]
where $j\geq 1$. The lacunarity assumption entails that the open balls $\opball{\lambda_{n}}{C'|\lambda_{n}|}$, for $n\geq 1$, are disjoint. Therefore, the balls $\opball{\lambda_{n}}{C' 2^{j-1}}$ indexed by $n\in\Ncal_{j}(\lambda)$ do not intersect either. Meanwhile, all these balls are included in the open ball centered at the origin with radius $(1+C'/2)2^{j}$. Comparing the volume of these balls, we infer that $\card\Ncal_{j}(\lambda)$ is bounded above by $(1+2/C')^{d}$. In addition, the set
\[
\Ncal_{0}(\lambda)=\{n\geq 1 \:|\: |\lambda_{n}|<1\}
\]
is necessarily finite; in fact, as a result of the separateness condition, the balls $\opball{\lambda_{n}}{C/2}$, for $\Ncal_{0}(\lambda)$, are disjoint and included in the open ball centered at zero with radius $1+C/2$, and the same volume comparison argument implies that $\card\Ncal_{0}(\lambda)$ is at most $(1+2/C)^{d}$. As a consequence, the sequence $\lambda$ is sparse.

The above approach also enables us to write the sparse set $\{\lambda_{n},\,n\geq 1\}$ as a finite union of sets of the form $\{\lambda^{(k)}_{n},\,n\geq 1\}$, where each sequence $(\lambda^{(k)}_{n})_{n\geq 1}$ is a {\em Hadamard sequence}, which means that it is separated and satisfies
\[
\exists C''>1 \quad \forall n\geq 1 \qquad \frac{|\lambda^{(k)}_{n+1}|}{|\lambda^{(k)}_{n}|}\geq C''.
\]
Each of these sequences is obtained first by considering the even values of $j$ and retaining only one term of the initial sequence $\lambda$ among those indexed by $\Ncal_{j}(\lambda)$, and then by handling the odd values of $j$. Note that a Hadamard sequence is clearly lacunary, and therefore sparse, but the converse need not hold. In addition, a finite union of Hadamard sequences need not be lacunary.

\subsection{Decay of sequences with sparse support and behavior of the jump operator}

Recall that the spaces $\Fcal^{\gamma}$ play an important role in the analysis of the jump operator $J$; in fact, we proved in Section~\ref{sec:jumpop} that the jump operator $J$ is a bicontinuous automorphism of the spaces $\Fcal^{\gamma}$, for $\gamma>1$. We shall now obtain an analogous result in the case where $0<\gamma\leq 1$, up to a sparsity assumption. To be precise, the set $\Fcal^{\gamma}$ will be replaced by the subspace $\Fcal^{\gamma}_{\Scal}$ formed by the sequences with sparse support, namely,
\[
\Fcal^{\gamma}_{\Scal}=\Scal\cap\Fcal^{\gamma},
\]
where $\Scal$ denotes the vector space of all odd sequences $a=(a_{n})_{n\in\Z^{d}}$ for which the support $\supp{a}$ is a sparse subset of $\Z^{d}$.

Before studying the behavior of the jump operator on the spaces $\Fcal^{\gamma}_{\Scal}$, let us point out that they may be used to characterize the decay of a sequence $a\in\Scal$. Specifically, we measure the rate of decay of such a sequence by considering
\begin{equation}\label{eq:defgamma}
\gamma_{a}=\liminf_{n\to\infty \atop n\in\supp{a}}\frac{-\log|a_{n}|}{\log|n|}.
\end{equation}
In particular, a sequence $a\in\Scal$ has slow decay on its support in the sense of Definition~\ref{df:slowdecay} if and only if $\gamma_{a}$ vanishes. One then easily checks that
\[
\gamma_{a}=\sup\{\gamma>0\:|\:a\in \Fcal^{\gamma}_{\Scal}\}.
\]
It will also be useful to remark that, due to the sparsity of the support, $\gamma_{a}$ may as well be seen as a critical exponent for the convergence of a series. Indeed, the fact that $a$ is in $\Scal$ implies that
\begin{equation}\label{eq:defgammabis}
\gamma_{a}=\sup\left\{\gamma>0\:\Biggl|\:\sum_{n\in\Z^{d}}|n|\,|a_{n}|^{1/\gamma}<\infty\right\}
=\inf\left\{\gamma>0\:\Biggl|\:\sum_{n\in\Z^{d}}|n|\,|a_{n}|^{1/\gamma}=\infty\right\}.
\end{equation}
As we shall show below, the exponent $\gamma_{a}$ will play a crucial role in the study of the multifractal properties of the Davenport series with coefficients given by $a$, see Corollary~\ref{cor:bndgamma} below as well as the results of Section~\ref{sec:multifrac}.

Let us now describe the action of $J$ on the spaces $\Fcal^{\gamma}_{\Scal}$. The next result may be seen as a partial extension of Proposition~\ref{prp:equivdec} to the case where $\gamma$ is no more restricted to be larger than one. However, it is weaker because it does not discuss invertibility properties and the target set of the jump operator $J$ is the intersection
\begin{equation}\label{eq:defFgammamin}
\Fcal^{\gamma,-}=\bigcap_{\eps>0} \Fcal^{\gamma-\eps}
\end{equation}
instead of the mere space $\Fcal^{\gamma}$. Note that $\Fcal^{\gamma,-}$ is endowed with the natural Fr\'echet topology inherited from the norms on the spaces $\Fcal^{\gamma-\eps}$.

\begin{prp}\label{prp:equivdec2}
For any $\gamma>0$, the jump operator $J$ induces a continuous mapping in the following way:
\[
\Fcal^{\gamma}_{\Scal}
\xymatrix{
\ar[r]^{J} &
}
\Fcal^{\gamma,-}
\]
\end{prp}

\begin{proof}
Given $\gamma>0$, let $a=(a_{n})_{n\in\Z^{d}}$ be a sequence in $\Fcal^{\gamma}_{\Scal}$, and let $A=(A_{q})_{q\in\Z^{d}}$ denote its image under $J$, that is, $A=J(a)$. Then, for each vector $q\in\ZD$,
\[
|A_{q}|\leq 2\sum_{l=1}^{\infty} |a_{lq}|\leq 2|a|_{\Fcal^{\gamma}}\sum_{l=1}^{\infty}\frac{\ind_{\{lq\in\supp{a}\}}}{|lq|^{\gamma}}.
\]
In order to give an upper bound on the last sum, we split the index set into dyadic intervals. For each integer $j\geq 0$, we have
\[
\sum_{l=2^{j}}^{2^{j+1}-1}\frac{\ind_{\{lq\in\supp{a}\}}}{|lq|^{\gamma}}\leq\frac{2^{-\gamma j}}{|q|^{\gamma}}\card(\supp{a}\cap\opball{0}{|q|2^{j+1}}).
\]
The support of the sequence $a$ is sparse; thus, for all $\eps>0$, its intersection with the open ball centered at the origin with radius $|q|2^{j+1}$ has cardinality at most $C_{\eps}|q|^{\eps}2^{\eps j}$ for some real $C_{\eps}>0$ that depends on neither $q$ nor $j$. We deduce that
\[
|A|_{\Fcal^{\gamma-\eps}}=\sup_{q\in\Z^{d}}|q|^{\gamma-\eps}|A_{q}|\leq 2|a|_{\Fcal^{\gamma}}C_{\eps}\sum_{j=0}^{\infty}2^{(\eps-\gamma)j}=\frac{2C_{\eps}}{1-2^{\eps-\gamma}}|a|_{\Fcal^{\gamma}},
\]
with the proviso that $\eps<\gamma$. The result follows.
\end{proof}

Note that, even when a sequence $a$ has sparse support, the support of the associated sequence $J(a)$ of jump sizes need not be sparse; this is why the target set in the above statement involves $\Fcal^{\gamma-\eps}$, but not $\Fcal^{\gamma-\eps}_{\Scal}$. Moreover, the jump operator on $\Fcal^{\gamma}_{\Scal}$ entails a slight loss in the speed of decay in the sense that the target set is not exactly $\Fcal^{\gamma}$, as in Proposition~\ref{prp:equivdec}, but rather $\Fcal^{\gamma,-}$. Still, a simple adaptation of the above proof shows that $J(a)$ actually belongs to $\Fcal^{\gamma}$ when $a$ is a sequence in $\Fcal^{\gamma}$ for which $\supp{a}$ is composed by the values of a separated and lacunary sequence, which is stronger than assuming that $a$ is in $\Scal$.
 
\subsection{Pointwise regularity of sparse Davenport series}

Now that the notion of sequence with sparse support has been defined, we are in position to introduce the notion of sparse Davenport series.

\begin{df}
A Davenport series with coefficients given by a sequence $a\in\ell^{1}$ is sparse if the support $\supp{a}$ is a sparse set, that is, $a\in\Scal\cap\ell^{1}$.
\end{df}

In order to recover the regularity of the Davenport series $f$ at every point, we shall assume, in addition to the sparsity of the support, that there is no cancellation in the sums~(\ref{eq:defAq}) defining the jump operator, in the sense that $A_{q}$ is at least of the order of magnitude of its largest term. To be specific, for any sequence $a=(a_{n})_{n\in\Z^{d}}$ in $\ell^{1}$, we may consider the even sequence $\overline{a}=(\overline{a}_{q})_{q\in\Z^{d}}$ in $\ell^{\infty}(\Z^{d})$ given by
\[
\forall q\in\Z^{d} \qquad \overline{a}_{q}=\sup_{l\geq 1}|a_{lq}|.
\]
This enables us to define in the following manner a sublinear operator $M$ on $\ell^{1}$, which we refer to as the maximal operator:
\[
\forall (a_{n})_{n\in\Z^{d}}\in\ell^{1} \qquad M((a_{n})_{n\in\Z^{d}})=(\overline{a}_{q})_{q\in\Z^{d}}.
\]
The jump operator and the maximal operator both act on the sequences $a=(a_{n})_{n\in\Z^{d}}$ in $\ell^{1}$, and the asymptotic behavior of these actions may be compared by means of
\[
\theta_{a}=\limsup_{q\to\infty \atop q\in\supp{\overline{a}}} \frac{\log|A_q|}{\log|\overline{a}_{q}|},
\]
where $A$ and $\overline{a}$ denote the sequences $J(a)$ and $M(a)$, respectively. Incidentally, it is useful to remark that
\begin{equation}\label{eq:inclMJ}
\supp{a}\cup\supp{A}\subseteq\supp{\overline{a}},
\end{equation}
and that $\supp{A}$ and $\supp{\overline{a}}$ coincide asymptotically whenever $\theta_{a}$ is finite, in the sense that they differ by a finite number of points only. The aforementioned assumption may now be expressed by means of the following definition.

\begin{df}
A Davenport series with coefficients given by a sequence $a\in\ell^{1}$ is asymptotically jump canceling if $\theta_{a}>1$.
\end{df}

More precisely, assuming that there is no cancellation in the sums defining the jumps sizes $A_{q}$ amounts to supposing that $\theta_{a}$ is bounded above by one, {\em i.e.}~that the Davenport series is not jump canceling. The next result shows that, in that situation, the upper bound given by Corollary~\ref{cor:upbndjumpDav} becomes an equality and may actually be replaced by an expression that is easier to handle. Specifically, the jump sizes $A_{q}$ arising in the bound may be replaced by the Davenport coefficients $a_{n}$ themselves, and the distance $\delta^{\Pcal}_{q}(x_{0})$ may be replaced by
\begin{equation}\label{eq:defdeltanx}
\delta_{n}(x_{0})=\dist\left(x_{0},\bigcup_{k\in\Z} H_{k,n}\right)=\frac{1}{|n|}\inf_{k\in\Z} |n\cdot x_{0}-k|,
\end{equation}
which means that we may discard the rather complicated coprimeness condition arising in~(\ref{eq:defdPqx0}) and discussed in Section~\ref{sec:upbnd}.

Recall that the H\"older exponent vanishes wherever the Davenport series is not continuous, {\em i.e.}~on the set defined by~(\ref{eq:defdisc}). This set is a union of hyperplanes which may be written in the form $\disc{J(a)}$, where for any odd sequence $b=(b_{q})_{q\in\Z^{d}}$,
\begin{equation}\label{eq:defDb}
\disc{b}=\bigcup_{(p,q)\in\Hcal_{d} \atop q\in\supp{b}} H_{p,q}.
\end{equation}
We may therefore restrict our attention to the points at which the series is continuous, {\em i.e.}~outside the set $\disc{J(a)}$. Actually, our approach only enables us to recover the H\"older exponent of the Davenport series outside the set $\disc{M(a)}$, which may be larger than $\disc{J(a)}$ in view of~(\ref{eq:inclMJ}). Yet, this slight restriction will not prevent us from performing the multifractal analysis of the Davenport series that are not jump canceling, because the two sets then differ by a finite number of hyperplanes only. In the next statement, as before, $A=(A_{q})_{q\in\Z^{d}}$ denotes the image of $a$ under the jump operator $J$.

\begin{thm}\label{thm:formhold}
Let $f$ be a Davenport series with coefficients given by a sequence $a\in\ell^{1}$. Let us assume that the series is sparse and not asymptotically jump canceling, that is,
\[
a\in\Scal \qquad\mbox{and}\qquad \theta_{a}\leq 1.
\]
Then, the H\"older exponent of $f$ at any fixed point $x_{0}\in\R^{d}$ satisfies
\begin{equation}\label{eq:bndholdAa}
h_{f}(x_{0})
\leq\liminf_{q\to\infty\atop q\in\supp{A}}\frac{\log|A_{q}|}{\log\delta^{\Pcal}_{q}(x_{0})}
\leq\liminf_{n\to\infty\atop n\in\supp{a}}\frac{\log|a_{n}|}{\log\delta_{n}(x_{0})}.
\end{equation}
Moreover, if $x_{0}$ does not belong to $\disc{M(a)}$, the above quantities coincide and the H\"older exponent may be computed using either of the following two formulae:
\[
h_{f}(x_{0})=\liminf_{q\to\infty\atop q\in\supp{A}}\frac{\log|A_{q}|}{\log\delta^{\Pcal}_{q}(x_{0})}
\qquad\mbox{and}\qquad
h_{f}(x_{0})=\liminf_{n\to\infty\atop n\in\supp{a}}\frac{\log|a_{n}|}{\log\delta_{n}(x_{0})}.
\]
\end{thm}

The proof of Theorem~\ref{thm:formhold} is postponed to Section~\ref{sec:proofthmformhold} for the sake of clarity. In the course of the proof, we obtain a uniform bound on the H\"older exponent, which may be seen as a consequence of Theorem~\ref{thm:formhold}, and which we state now as a separate result for future reference.

\begin{cor}\label{cor:bndgamma}
Let $f$ be a Davenport series with coefficients given by a sequence $a=(a_{n})_{n\in\Z^{d}}$ in $\ell^{1}$. If the series is sparse and not asymptotically jump canceling, then
\[
\forall x_{0}\in\R^{d} \qquad h_{f}(x_{0})\leq\gamma_{a},
\]
where $\gamma_{a}$ is defined by~(\ref{eq:defgamma}). In particular, the H\"older exponent of $f$ vanishes everywhere when the sequence $a$ has slow decay.
\end{cor}

\section{Implications for multifractal analysis}\label{sec:multifrac} 

The preceding results will allow us to perform the multifractal analysis of some multivariate Davenport series with coefficients $a\in \ell^{1}$. From now on, {\em we assume that the series is sparse and not asymptotically jump canceling}. We begin by describing the size properties of the iso-H\"older sets, which are formed of the points where $f$ has H\"older exponent equal to a given $h$, specifically,
\begin{equation}\label{eq:dfisohold}
E_{f}(h)=\{x\in\R^{d}\:|\:h_{f}(x)=h\},
\end{equation}
for $h\in[0,\infty]$. To be precise, we compute the local spectrum of singularities of the series $f$, that is, the mapping
\begin{equation}\label{eq:dflocspec}
d_{f}(h,W)=\Hdim(E_{f}(h)\cap W),
\end{equation}
where $W$ is a nonempty open subset of $\R^{d}$. In the previous formula, $\Hdim$ denotes Hausdorff dimension, whose definition is recalled in Section~\ref{subsec:sliLaalpha}. The spectrum is actually governed by the parameter $\gamma_{a}$ which controls the decay of the sequence $a$ and is defined by~(\ref{eq:defgamma}).

In view of Corollary~\ref{cor:bndgamma}, the Davenport series $f$ has H\"older exponent at most $\gamma_{a}$ everywhere. Thus, all the iso-H\"older sets $E_{f}(h)$, for $h>\gamma_{a}$, are empty. As a consequence, the spectrum of singularities is supported on $[0,\gamma_{a}]$, and we may restrict our attention to that interval in what follows. In addition, if $\gamma_{a}$ vanishes, {\em i.e.}~when the sequence $a$ has slow decay, then the H\"older exponent of $f$ vanishes everywhere, so that all the iso-H\"older sets are empty, except $E_{f}(0)$ which is equal to the whole space $\R^{d}$. That situation being trivial, we may assume from now on that {\em $\gamma_{a}$ is positive}.

The analysis below does not cover the case where $\gamma_{a}$ is infinite. Note that this case includes that in which the sequence $a$ has finite support. In that situation, the Davenport series is a finite sum of piecewise linear functions, thereby being smooth except on a locally finite union of hyperplanes where its H\"older exponent vanishes. If $\gamma_{a}$ is infinite and the support of $a$ has infinite cardinality, the arguments below only imply that the H\"older exponent of the Davenport series is infinite Lebesgue-almost everywhere in $\R^{d}$, and that the iso-H\"older sets associated with finite values of the exponent all have Hausdorff dimension at most $d-1$. It seems plausible, though, that the dimension is exactly $d-1$. In what follows, we shall therefore assume that {\em $\gamma_{a}$ is finite}.

The local spectrum of singularities of $f$ on the interval $[0,\gamma_{a}]$ is then given by the following statement. 

\begin{thm}\label{thm:spec}
Let $f$ be a Davenport series with coefficients given by a sequence $a\in \ell^{1}$. Let us assume that the series is sparse and not asymptotically jump canceling, and that $0<\gamma_{a}<\infty$. Then, for any real $h\in[0,\gamma_{a}]$ and any nonempty open subset $W$ of $\R^{d}$,
\[
d_{f}(h,W)=d-1+\frac{h}{\gamma_{a}}.
\]
\end{thm}

Note that the spectrum of singularities does not depend on the particular region $W$ that is considered, and moreover it is nondegenerate, in the sense that its support is not reduced to a single point. Consequently, following the terminology of~\cite{Jaffard:2004fh}, the Davenport series falls in the category of {\em homogeneous multifractal functions}.

We get comparable results for the singularity sets, which are composed by the points where the Davenport series $f$ is continuous and has H\"older exponent at most a given $h$, that is,
\[
E'_{f}(h)=\{x\in\R^{d}\setminus\disc{J(a)}\:|\:h_{f}(x)\leq h\},
\]
where $\disc{J(a)}$ is the set of discontinuities of $f$, given by~(\ref{eq:defdisc}) and written using the notation~(\ref{eq:defDb}). In addition, we prove that the singularity sets belong to the category of sets with large intersection introduced by K.~Falconer~\cite{Falconer:1994hx}. This remarkable property essentially asserts that the sets are so omnipresent and large in a measure theoretic sense that their size properties are not altered by taking countable intersections. As a matter of fact, the intersection of countably many sets with large intersection with Hausdorff dimension at least a given real $s$ still has dimension at least $s$; this is in stark contradiction with the fact that the codimension of the intersection of two subsets is usually expected to be the sum of their codimension, as is the case for affine subspaces. Formally, the class of sets with large intersection are defined in~\cite{Falconer:1994hx} in the following manner. Recall that a $G_{\delta}$-set is a set that may be expressed as a countable intersection of open sets.

\begin{df}\label{def:licfalc}
For any $s\in(0,d]$, the class $\licfalc{s}$ of sets with large intersection with dimension at least $s$ is defined as the collection of all $G_{\delta}$-subsets $E$ of $\R^{d}$ such that
\[
\Hdim\bigcap_{n\geq 1}\varsigma_{n}(E)\geq s
\]
for any sequence $(\varsigma_{n})_{n\geq 1}$ of similarity transformations of $\R^{d}$.
\end{df}

The class $\licfalc{s}$ is closed under countable intersections and bi-Lipschitz transformations, and is the maximal class of $G_{\delta}$-sets with Hausdorff dimension at least $s$ that satisfies those properties, see~\cite[Theorem~A]{Falconer:1994hx} for a precise statement. Moreover, every set of the class $\licfalc{s}$ has packing dimension equal to $d$ in every nonempty open set, see~\cite[Theorem~D]{Falconer:1994hx}. In what follows, packing dimension is denoted by $\Pdim$; we refer for example to~\cite{Falconer:2003oj} for a definition of this notion.

Restricting to $G_{\delta}$-sets will be quite a constraint for us here, so instead of considering the classes $\licfalc{s}$ themselves, we shall work with the extended classes $\varlicfalc{s}$ defined by the following condition: For all $E\subseteq\R^{d}$,
\[
E\in\varlicfalc{s} \qquad\Longleftrightarrow\qquad \exists E'\in\licfalc{s} \quad E'\subseteq E.
\]
It is clear that the class $\varlicfalc{s}$ contains $\licfalc{s}$ and, in view of~\cite[Theorem~C(b)]{Falconer:1994hx}, that the $G_{\delta}$-sets that belong to $\varlicfalc{s}$ actually belong to the original class $\licfalc{s}$. Moreover, the extended class $\varlicfalc{s}$ naturally inherits from $\licfalc{s}$ its remarkable properties: $\varlicfalc{s}$ is composed of sets with Hausdorff dimension at least $s$ and packing dimension equal to $d$, and is closed under countable intersections and bi-Lipschitz transformations.

The next result describes the size and large intersection properties of the singularity sets of the Davenport series $f$.

\begin{thm}\label{thm:slisingsets}
Let $f$ be a Davenport series with coefficients given by a sequence $a=(a_{n})_{n\in\Z^{d}}$ in $\ell^{1}$. Let us assume that the series is sparse and not asymptotically jump canceling, and that $0<\gamma_{a}<\infty$. Then, for any real $h\in(0,\gamma_{a}]$,
\[
E'_{f}(h)\in\varlicfalc{d-1+h/\gamma_{a}}
\]
and, moreover, for any nonempty open subset $W$ of $\R^{d}$,
\[
\Hdim(E'_{f}(h)\cap W)=d-1+\frac{h}{\gamma_{a}} \qquad\mbox{and}\qquad \Pdim(E'_{f}(h)\cap W)=d.
\]
\end{thm}

We refer to Section~\ref{sec:proofthmmultifrac} for the proof of the two above theorems.

\section{Convergence and global regularity of Davenport series}\label{sec:sobol}

We will now give a few results concerning the convergence of Davenport series, when the sequence of coefficients does not belong to $\ell^{1}$. In that case, the sum does not necessarily belong to $L^{\infty}$, so that H\"older pointwise regularity may not be a relevant notion anymore. We will mainly consider convergence in Sobolev spaces, with both positive and negative indices, which allows us to consider simultaneously convergence in spaces of functions or, more generally, distributions. Specific additional motivations for this section are supplied in Section~\ref{sec:conclud}, where we show that the determination of global Sobolev regularity exponents are a preliminary step to either the determination of $L^{q}$ regularity (which is needed for the study of $p$-exponents, see Section~\ref{subsec:conclud3}) or the verification of the multifractal formalism (see the beginning of Section~\ref{sec:conclud}).

\subsection{Preliminaries on multivariate arithmetic functions}

An arithmetic function is traditionally a mapping defined on $\N$ and valued in $\R$ or sometimes in $\C$. The usual multivariate extension deals with functions that are defined on $\N^{d}$, see~\cite{Breteche:2001ys,Essou}. In this paper, we consider a slightly different setting, with {\em multivariate arithmetic functions} defined on $\ZD$.

A first simple example is supplied by the natural extension to $\ZD$ of the divisor function already mentioned in Section~\ref{sec:jumpop}; this extension is still denoted by $\tau$ for simplicity. To be specific, for any $m\in \ZD$, we define $\tau(m)$ as the number of decompositions
\begin{equation}\label{eq:decompmln}
m=ln \qquad\mbox{with}\qquad l\in\N \quad\mbox{and}\quad n\in\ZD.
\end{equation}
It is clear that $\tau(m)$ coincides with $\tau(\gcd(m))$, where $\gcd(m)$ denotes the greatest common divisor of the components of the vector $m$. With the help of~(\ref{eq:asymptau}), this implies that $\tau(m)={\rm o}(|m|^{\eps})$ as $m$ goes to infinity, for any fixed $\eps>0$. In what follows, we shall write indistinctly $l|m$ and $n|m$ when~(\ref{eq:decompmln}) holds; with a slight abuse, we shall also write $l=m/n$.

We will also make use of the extensions to the multivariate setting of other arithmetic functions, specifically, the sums of $z$-th powers of the divisors. Given $z\in\C$, recall that the one-dimensional arithmetic function $\sigma_{z}$ is defined by
\[
\sigma_{z}(m)=\sum_{n|m} n^{z},
\]
where the sum bears on the positive divisors of the integer $m$. In the multivariate case, we have to draw a difference between integer and vector divisors. Therefore, we define two functions of the vectors $m\in\ZD$ by
\[
\sigma_{z}(m)=\sum_{n\in\ZD \atop n|m} |n|^{z}
\qquad\mbox{and}\qquad
\widetilde{\sigma}_{z}(m)=\sum_{l\in\N \atop l|m} l^{z}.
\]
It is clear that these two functions coincide on $\N$, and that $\sigma_{0}(m)=\widetilde{\sigma}_{0}(m)=\tau(m)$ for all $m\in\ZD$. Moreover, for any $z\in\C$, one easily checks that $\widetilde{\sigma}_{z}(m)=\sigma_{z}(\gcd(m))$ and $\sigma_{z}(m)=|m|^{z}\widetilde{\sigma}_{-z}(m)=|m|^{z}\sigma_{-z }(\gcd(m))$. Given that $\sigma_{z}(l)=l^{z}\sigma_{-z}(l)$ for any integer $l\in\N$, we deduce that
\begin{equation}\label{eq:relsig}
\forall m\in\ZD \qquad \sigma_{z}(m)=\left(\frac{|m|}{\gcd(m)}\right)^{z} \sigma_{z}(\gcd(m)).
\end{equation}

Finally, recall that the Dirichlet convolution of two arithmetic functions $A$ and $B$ defined on $\N$ is the arithmetic function $A\ast B$ given by
\[
\forall m\in\N \qquad A\ast B(m)=\sum_{(l,n)\in\N\times\N \atop ln=m} A(n)B(l).
\]
Similarly, the convolution of a {\em multivariate} arithmetic functions $A$ defined on $\ZD$ and a {\em one-dimensional} arithmetic function $B$ defined on $\N$ is the multivariate arithmetic function given by
\[
\forall m\in\ZD \qquad A\ast B(m)=\sum_{(l,n)\in\N\times\ZD \atop ln=m} A(n)B(l).
\]

\subsection{Davenport expansions vs Fourier expansions} 

Let us now go back to Davenport series. Without any assumption on the sequence $a=(a_{n})_{n\in\Z^{d}}$ of Davenport coefficients, the right-hand side of~(\ref{eq:davfou}) can be inverted using the {\em multivariate M\"obius inversion formula}, which is an easy extension of the one-dimensional case and calls upon the M\"obius function $\mu$ already used in Section~\ref{sec:jumpop}. However, we start by proving it for the sake of completeness.

\begin{lem}\label{lem:invmob}
Let $f$ be a multivariate arithmetic function defined on $\ZD$ and let $g$ be the multivariate arithmetic function given by
\[
\forall m\in\ZD \qquad g(m)=\sum_{n\in\ZD \atop n|m} f(n).
\]
Then, the function $f$ can be recovered from $g$ by $f=g\ast\mu$, {\em i.e.}
\[
\forall m\in\ZD \qquad f(m)=\sum_{n\in\ZD \atop n|m} g(n)\mu\left(\frac{m}{n}\right).
\]
\end{lem}

\begin{proof}
For any vector $m\in\ZD$, we have
\[
\sum_{n\in\ZD \atop n|m} g(n)\mu\left(\frac{m}{n}\right)
=\sum_{n\in\ZD \atop n|m} \mu\left(\frac{m}{n}\right)\sum_{k\in\ZD \atop k|n} f(k)
=\sum_{k\in\ZD \atop k|m} f(k)\sum_{n\in\ZD \atop k|n|m} \mu\left(\frac{m}{n}\right)
\]
Let us observe that the integer vectors $n\in\ZD$ satisfying $k|n|m$ are merely of the form $n=kl$, where $l$ ranges over the divisors of the positive integer $m/k$. Thus, the last sum satisfies
\[
\sum_{n\in\ZD \atop k|n|m}\mu\left(\frac{m}{n}\right)
=\sum_{l|(m/k)}\mu\left(\frac{m/k}{l}\right)
=\sum_{l|(m/k)}\mu(l)
=\ind_{\{k=m\}}.
\]
The last equality follows from the well-known fact that the sum of the M\"obius function over all positive divisors of a given natural number $n$ vanishes except if $n=1$, where the sum is equal to one. The result follows.
\end{proof}

The next proposition results from applying to~(\ref{eq:davfou}) the above inversion formula, and will be useful in the determination of the Sobolev regularity of Davenport series. It shows how to recover the Davenport coefficients of a series from the knowledge of its Fourier coefficients.

\begin{prp}
Let $f$ be a Davenport series with coefficients given by a sequence $a=(a_{n})_{n\in\Z^{d}}$ in $\ell^{1}$, and let $(c_{m})_{m\in\Z^{d}}$ denote the sequence of its Fourier coefficients. Then,
\[
\forall n\in\ZD \qquad a_{n}=-\pi\sum_{m\in\ZD \atop m|n} \frac{m}{n}\mu\left(\frac{n}{m}\right)c_{m}.
\]
\end{prp}

\begin{proof}
A straightforward consequence of~(\ref{eq:davfou}) is that for all $m\in\ZD$,
\[
-\pi c_{m}m=\sum_{n\in\ZD \atop n|m} a_{n}n.
\]
The result now follows from applying Lemma~\ref{lem:invmob} to the arithmetic functions $f(n)=a_{n}n$ and $g(m)=-\pi c_{m}m$. Note that these functions take values in $\R^d$ and not merely in $\R$. However, Lemma~\ref{lem:invmob} obviously extends to this case.
\end{proof}

\subsection{Regularity of the sum of a Davenport series}

Without any assumption on the odd sequence $a=(a_{n})_{n\in\Z^{d}}$ of Davenport coefficients, we may define an odd sequence $(c_{m})_{m\in\Z^{d}}$ with the help of~(\ref{eq:davfou}). This detour via Fourier series will allow us to study the convergence of the Davenport series $\sum_{n} a_{n}\{n\cdot x\}$, even when $a$ is no longer assumed to belong to $\ell^{1}$. Indeed, we shall see that, in many functional settings, when the associated Fourier series $\sum_{m} c_{m}\sin(2\pi m\cdot x)$ converges, then the partial sums of the Davenport series converge to the same limit.

In order to be more precise, let us begin by recalling that the spaces $\Fcal^{\gamma,-}$ are defined in terms of the sequence spaces $\Fcal^{\gamma}$ by means of~(\ref{eq:defFgammamin}). Moreover, let $F^{\gamma}$ denote the space of distributions whose Fourier coefficients belong to $\Fcal^{\gamma}$, and by $F^{\gamma,-}$ the space of distributions whose Fourier coefficients belong to $\Fcal^{\gamma,-}$.

In addition, for any odd sequence $a=(a_{n})_{n\in\Z^{d}}$, we denote by $f^{N}$ the partial sums of the corresponding Davenport series, that is,
\[
f^{N}(x)=\sum_{n\in\Z^{d} \atop |n|\leq N} a_{n}\{n\cdot x\}.
\]
The next result discusses the convergence properties of the sequence $(f^{N})_{N\geq 1}$ in the spaces $F^{\gamma}$ and $F^{\gamma,-}$. A noteworthy consequence lies in the fact that the Davenport series $\sum_{n} a_{n}\{n\cdot x\}$ converges in the sense of distributions when the coefficients $a_{n}$ do not increase faster than any polynomial.

\begin{prp}\label{prp:convpartialsums}
Let $\gamma\in\R$, and let $a\in\Fcal^{\gamma}$.
\begin{itemize}
\item If $\gamma<0$, then the sequence $(f^{N})_{N\geq 1}$ converges in $F^{\gamma,-}$ to a distribution $f$ which belongs to $ F^{\gamma}$.
\item If $0\leq\gamma\leq 2$, then the sequence $(f^{N})_{N\geq 1}$ is convergent in $F^{\min\{1,\gamma\},-}$.
\item if $\gamma>2$, then the sequence $(f^{N})_{N\geq 1}$ converges in $F^{1,-}$ to a distribution $f$ which belongs to $ F^{1}$.
\end{itemize}
\end{prp}

\begin{proof}
It follows from~(\ref{eq:davfou}) that the Fourier coefficients of the partial sum $f^{N}$ are given by
\[
c^{N}_{m}=-\frac{1}{2\pi}\sum_{(l,n)\in\Z^{\ast}\times\ZD \atop ln=m, \; |n|\leq N}\frac{a_{n}}{l}.
\]
The condition $|n|\leq N$ is necessarily satisfied as soon as $N$ is greater than or equal to $|m|$, so that each sequence $(c^{N}_{m})_{N\geq 1}$ is ultimately constant equal to
\[
c_{m}=-\frac{1}{2\pi}\sum_{(l,n)\in\Z^{\ast}\times\ZD \atop ln=m}\frac{a_{n}}{l}.
\]
Moreover, given that the sequence $a$ belongs to the space $\Fcal^{\gamma}$, it is easy to check that for all $N\geq 1$ and $m\in\ZD$,
\[
|c^{N}_{m}|\leq\frac{|a|_{\Fcal^{\gamma}}}{\pi |m| }\sigma_{1-\gamma}(m),
\]
which implies in particular that for all $m\in\ZD$,
\begin{equation}\label{eq:estsig} 
|c_{m}|\leq\frac{|a|_{\Fcal^{\gamma}}}{\pi |m|}\sigma_{1-\gamma}(m).
\end{equation}

The following estimates on the one-variable arithmetic functions $\sigma_{z}(m)$ for $z\in\R$ may be found in~\cite{Sandor:1996fk}:
\[
\left\{\begin{array}{ccl}
z<-1 & \qquad\Longrightarrow\qquad & \sigma_{z}(m)={\rm O}(1) \\[1mm]
-1\leq z<0 & \qquad\Longrightarrow\qquad & \forall \eps>0 \quad \sigma_{z}(m)={\rm O}(m^{\eps}) \\[1mm]
0\leq z\leq 1 & \qquad\Longrightarrow\qquad & \forall \eps>0 \quad \sigma_{z}(m)={\rm O}(m^{z+\eps}) \\[1mm]
z>1 & \qquad\Longrightarrow\qquad & \sigma_{z}(m)={\rm O}(m^{z}).
\end{array}\right.
\]
Thanks to~(\ref{eq:relsig}), we easily deduce the following estimates on the corresponding multivariate functions:
\begin{equation}\label{eq:estimsig}
\left\{\begin{array}{ccl}
z<-1 & \qquad\Longrightarrow\qquad & \sigma_{z}(m)={\rm O}(1) \\[1mm]
-1\leq z<0 & \qquad\Longrightarrow\qquad & \forall \eps>0 \quad \sigma_{z}(m)={\rm O}(|m|^{\eps}) \\[1mm]
0\leq z\leq 1 & \qquad\Longrightarrow\qquad & \forall \eps>0 \quad \sigma_{z}(m)={\rm O}(|m|^{z+\eps}) \\[1mm]
z>1 & \qquad\Longrightarrow\qquad & \sigma_{z}(m)={\rm O}(|m|^{z}).
\end{array}\right.
\end{equation}
It now follows from (\ref{eq:estimsig}) that the $|c^{N}_{m}|$ satisfy the estimates of Proposition~\ref{prp:convpartialsums} uniformly in $N$, and their limits $|c_{m}|$ satisfy the same estimates.

Convergence in the corresponding function spaces follows immediately by applying the same approach to the differences $c^{N}_{m}-c_{m}$, starting from the observation that for any fixed $\eps>0$,
\[
|c^{N}_{m}-c_{m}|\leq\frac{|a|_{\Fcal^{\gamma}}}{\pi |m| N^{\eps}}\sigma_{1-\gamma+\eps}(m)
\]
for all $N\geq 1$ and $m\in\ZD$.
\end{proof}

Our purpose is now to determine in which Sobolev spaces the Davenport series with coefficients in the space $\Fcal^{\gamma}$ do converge. Let us recall that the Sobolev space $H^{s}$ is characterized by the following condition on the Fourier coefficients: A $\Z^{d}$-periodic odd distribution $f$ belongs to $H^{s}$ if the sequence $(c_{m})_{m\in\Z^{d}}$ of its Fourier coefficients satisfies
\[
|f|_{H^{s}}^{2}=\sum_{m\in\ZD} |c_{m}|^{2} |m|^{2s}<\infty.
\]
Note that, if $s<0$, this defines a space of distributions. In order to state sharp results, we shall also need the following slight modifications of $H^{s}$. Specifically, let $H^{s}_{\delta}$ be the space of all $\Z^{d}$-periodic odd distributions $f$ whose Fourier coefficients satisfy
\[
|f|_{H^{s}_{\delta}}^{2}=\sum_{m\in\ZD} |c_{m}|^{2}\frac{|m|^{2s}}{(1+\log |m|)^{\delta}}<\infty,
\]
and let $H^{s}_{\delta,+}$ and $H^{s,-}$ be the spaces defined respectively by
\[
H^{s}_{\delta,+}=\bigcap_{\eps>0} H^{s}_{\delta+\eps} \qquad\mbox{and}\qquad H^{s,-}=\bigcap_{\eps>0} H^{s-\eps}.
\]

Before proceeding, let us begin by observing that the Fourier coefficient indexed by $m=(m_{1},\ldots,m_{d})\in\ZD$ of the function which maps $x= (x_{1},\cdots,x_{d})$ to $\{x_{1}\}$ is equal to $\ind_{\{m_{2}=\ldots=m_{d}=0\}}/m_{1}$. Therefore, this function fails to belong $H^{1/2}$ but belongs to $H^{1/2}_{1,+}$. It follows that, no matter how large $\gamma$ is picked, we cannot expect substantially better results than convergence in $H^{1/2}_{1,+}$. In addition, note that if $s$ is less than $1/2$, then
\[
|\{n\cdot x\}|_{H^{s}}=\frac{|n|^{s}}{2\pi}\left(2\zeta(2(1-s))\right)^{1/2}.
\]
Thus, for any odd sequence $a=(a_{n})_{n\in\Z^{d}}$, the Davenport series $f$ defined by~(\ref{eq:dave}) has norm $|f|_{H^{s}}$ bounded above by $\sum_{n} |a_{n}|\,|n|^{s}$, up to a multiplicative constant. As a consequence,
\[
\forall s<1/2 \qquad \sum_{n\in\ZD} |a_{n}|\,|n|^{s}<\infty \quad\Longrightarrow\quad f\in H^{s}.
\]
We will now see how this straightforward result can be improved under the assumption that $a$ belongs to the space $\Fcal^{\gamma}$. We restrict our attention to the situation where $d\geq 2$, the one-dimensional case being thoroughly studied in~\cite{Jaffard:2004os}.

\begin{prp}\label{prp:convpartialsumsbis}
Let us assume that $d\geq 2$. Let $\gamma\in\R$, and let $a=(a_{n})_{n\in\Z^{d}}$ be a sequence in $\Fcal^{\gamma}$. Then, the sequence $(f^{N})_{N\geq 1}$ converges in the space
\begin{equation}\label{eq:sobolevgamma}
\begin{cases}
H^{\gamma-d/2}_{1,+} & \mbox{if } \gamma\leq 0\\[1mm]
H^{\gamma-d/2,-} & \mbox{if } 0<\gamma\leq 1\\[1mm]
H^{(1+\gamma-d)/2,-} & \mbox{if } 1<\gamma\leq 2\\[1mm]
H^{(1+\gamma-d)/2}_{1,+} & \mbox{if } 2<\gamma<d \mbox{ and } d\geq 3\\[1mm]
H^{1/2}_{2,+} & \mbox{if } \gamma=d\geq 3\\[1mm]
H^{1/2}_{1,+} & \mbox{if } \gamma>d.
\end{cases}
\end{equation}
\end{prp}

\begin{proof}
It follows from Proposition~\ref{prp:convpartialsums} that the sequence $(f^{N})_{N\geq 1}$ converges to a distribution $f$. Moreover, the Fourier coefficients of $f$ may be bounded with the help of~(\ref{eq:estsig}), so that
\begin{equation}\label{eq:normsobolevf}
|f|_{H^{s}_{\delta}}^{2}\leq\frac{|a|_{\Fcal^{\gamma}}^{2}}{\pi^{2}}\sum_{m\in\ZD}\frac{|m|^{2(s-1)}}{(1+\log|m|)^\delta}\sigma_{1-\gamma}(m)^{2}.
\end{equation}
In order to estimate the above sum, let us split the index set as a union of dyadic domains. Specifically, the above sum is equal to the sum over all integers $j\geq 0$ of
\begin{equation}\label{eq:sumdyad}
\sum_{m\in\Z^{d} \atop 2^{j}\leq|m|<2^{j+1}}\frac{|m|^{2(s-1)}}{(1+\log|m|)^\delta}\sigma_{1-\gamma}(m)^{2}
\leq\frac{2^{2(s-1)(j+1)}}{(1+j)^{\delta}}\mfrak_{\gamma}(2^{j+1})\sfrak_{\gamma}(2^{j+1}).
\end{equation}
In the previous bound, $\mfrak_{\gamma}$ and $\sfrak_{\gamma}$ are defined respectively by
\[
\mfrak_{\gamma}(x)=\sup_{m\in\ZD \atop |m|<x}\sigma_{1-\gamma}(m)
\qquad\mbox{and}\qquad
\sfrak_{\gamma}(x)=\sum_{m\in\ZD \atop |m|<x}\sigma_{1-\gamma}(m),
\]
for any real $x>0$. The estimates~(\ref{eq:estimsig}) readily imply the following bounds on $\mfrak_{\gamma}$:
\[
\mfrak_{\gamma}(x)=
\begin{cases}
{\rm O}(x^{1-\gamma}) & \mbox{if } \gamma<0\\[1mm]
{\rm O}(x^{1-\gamma+\eps}) \mbox{ for all } \eps>0 & \mbox{if } 0\leq\gamma\leq 1\\[1mm]
{\rm O}(x^{\eps}) \mbox{ for all } \eps>0 & \mbox{if } 1<\gamma\leq 2\\[1mm]
{\rm O}(1) & \mbox{if } \gamma>2.
\end{cases}
\]
Let us now deal with $\sfrak_{\gamma}$. For all $x>0$, we have
\[
\sfrak_{\gamma}(x)
=\sum_{m\in\ZD \atop |m|<x}\sum_{n\in\ZD \atop n|m} |n|^{1-\gamma}
\leq\sum_{n\in\ZD \atop |n|<x} \frac{x}{|n|}|n|^{1-\gamma}
=x\sum_{n\in\ZD \atop |n|<x} |n|^{-\gamma},
\]
which readily implies that
\[
\sfrak_{\gamma}(x)
=\begin{cases}
{\rm O}(x^{d+1-\gamma}) & \mbox{if } \gamma<d\\[1mm]
{\rm O}(x\log x) & \mbox{if } \gamma=d\\[1mm]
{\rm O}(x) & \mbox{if } \gamma>d.
\end{cases}
\]

Combining the above bounds, we deduce that the sum in~(\ref{eq:sumdyad}) is bounded above, up to a multiplicative constant, by
\[
\begin{cases}
2^{(d+2(s-\gamma))j}/j^{\delta} & \mbox{if } \gamma<0\\[1mm]
2^{(d+2(s-\gamma)+\eps)j} \mbox{ for all } \eps>0 & \mbox{if } 0\leq\gamma\leq 1\\[1mm]
2^{(d+2s-1-\gamma+\eps)j} \mbox{ for all } \eps>0 & \mbox{if } 1<\gamma\leq 2\\[1mm]
2^{(d+2s-1-\gamma)j}/j^{\delta} & \mbox{if } 2<\gamma<d \mbox{ and } d\geq 3\\[1mm]
2^{(2s-1)j}/j^{\delta-1} & \mbox{if } \gamma=d\geq 3\\[1mm]
2^{(2s-1)j}/j^{\delta} & \mbox{if } \gamma>d.
\end{cases}
\]
These estimates are now sufficient to deduce that the limiting distribution $f$ belongs to the spaces given by~(\ref{eq:sobolevgamma}).

On top of that, convergence in the corresponding spaces follows immediately by applying the same approach to the differences $f^{N}-f$. As a matter of fact, for any fixed $\eta>0$, their Fourier coefficients satisfy
\[
|c^{N}_{m}-c_{m}|\leq\frac{|a|_{\Fcal^{\gamma}}}{\pi |m|}\left(\frac{1+\log|m|}{1+\log N}\right)^{\eta/2}\sigma_{1-\gamma}(m)
\]
for all $N\geq 1$ and $m\in\ZD$, which implies that~(\ref{eq:normsobolevf}) also holds when replacing $|f|_{H^{s}_{\delta}}^{2}$ by $(1+\log N)^{\eta}|f^{N}-f|_{H^{s}_{\delta+\eta}}^{2}$.
\end{proof}

\section{Concluding remarks and open problems}\label{sec:conclud}

The study of the local regularity of Davenport series remains a largely open field of investigations, with many interesting questions at the crossroad of number theory, harmonic analysis and functional analysis; our purpose in this section is to list a few of them that we believe of particular interest. A first one consists in the verification of the {\em multifractal formalism}. We shall not describe this question here, because its final and most precise formulation (in terms of {\em wavelet leaders}) requires the introduction of wavelet methods that go beyond the scope of the present paper; we refer to~\cite{Jaffard:2004fh} for a mathematical presentation concerning these issues, and to~\cite{AJW1} for a recent overview on the applications side. The verification of the multifractal formalism is a completely open problem for series of compensated pure jump functions, whether they be Davenport series (in one or several variables), or L\'evy processes and fields. Let us just mention that Section~\ref{sec:sobol} can be seen as a preliminary step in this direction; indeed, a part of this verification involves the determination of the Sobolev spaces that contain the function $f$ under consideration.

\subsection{Optimality of Lemma~\ref{lem:upbndjump}}\label{subsec:conclud1}

The only cases where we have been able to determine the exact pointwise H\"older regularity of the sum of a Davenport series are when the bound given by Lemma~\ref{lem:upbndjump} is optimal. This is not accidental and actually, in all cases of jump functions for which the H\"older exponent has been determined, it turns out that this bound is optimal: This is the case for L\'evy processes without Brownian component and their extension to the multivariate setting~\cite{Durand:2007fk,Durand:2010uq,Jaffard:1999fg}, for the few cases of Markov processes with nonstationary increments whose multifractal analysis has been performed~\cite{Barral:2010vn}, and for the other cases of Davenport series which can be worked out~\cite{Jaffard:2004os,Jaffard:2009ve,Jaffard:2010qf}. We shall however give below a simple example of Davenport series where this is not the case.

We shall restrict the discussion to the one-dimensional setting, which is easier to consider and is sufficient to explain why Lemma~\ref{lem:upbndjump} is not always sharp, even in the setting of Davenport series. Of course, in all generality,~(\ref{eq:lem:upbndjump}) is clearly not always sharp, as shown by the case in which $f$ is a continuous function, where the bound thus obtained is trivial. A natural class of functions for which one might expect optimality is supplied by {\em compensated pure jump functions}, that is, the functions $f$ whose derivative $f'$ in a distributional sense is of the form
\begin{equation}\label{eq:cpjf}
\sum_{n=1}^{\infty} \left(a_{n} \delta_{x_{n}} + c_n\right).
\end{equation}
In typical examples, the jump locations $x_{n}$ form a dense subset of the ambient space. However, even in the case where $f$ is a compensated pure jump function, the bound~(\ref{eq:lem:upbndjump}) need not be optimal, as shown by the following example of one-dimensional Davenport series:
\begin{equation}\label{eq:deffbeta}
f_{\beta}(x)=-\zeta(\beta)\{x\}+\sum_{n=1}^{\infty} \frac{\{nx\}}{n^{\beta}},
\end{equation}
where $\beta>3/2$. Indeed, the function $f_{\beta}$ is continuous at zero and jumps at every nonvanishing rational $p/q$ written in its irreducible form and the corresponding jump has magnitude $\Delta_{f_{\beta}}(p/q)$ equal to $\zeta(\beta)/q^{\beta}$. Thus, the bound~(\ref{eq:lem:upbndjump}) on its H\"older exponent at zero is realized by rational numbers of the form $1/q$, specifically,
\[
h_{f_{\beta}}(0)\leq\liminf_{q\to\infty}\frac{\log\Delta_{f_{\beta}}(1/q)}{\log(1/q)}=\beta.
\]
We shall prove the following result which shows that this bound is not optimal.

\begin{prp}\label{prp:regfbetazero}
Let $\beta$ be a real number larger than $3/2$ that is not an integer greater than or equal to $3$. Then, the value of the H\"older exponent of $f_{\beta}$ at zero is given by
\[
h_{f_{\beta}}(0)=\beta-1.
\]
\end{prp}

\begin{proof}
Given that the function $f_{\beta}$ is odd, it is sufficient to study the increment $f_{\beta}(x)-f_{\beta}(0)$ for positive values of $x$ only. If $x\in(0,1)$ and $n<1/x$, we have $\{nx\}=nx-1/2$. Letting $\lceil\,\cdot\,\rceil$ denote the ceiling function, we deduce that
\begin{align*}
f_{\beta}(x)-f_{\beta}(0)&=-\zeta(\beta)\left(x-\frac{1}{2}\right)+\sum_{n=1}^{\lceil 1/x\rceil-1} \frac{nx-1/2}{n^{\beta}}+\sum_{n=\lceil 1/x\rceil}^{\infty} \frac{\{nx\}}{n^{\beta}}\\
&=-\zeta(\beta)x+x\sum_{n=1}^{\lceil 1/x\rceil-1} \frac{1}{n^{\beta-1}}+\sum_{n=\lceil 1/x\rceil}^{\infty} \frac{\{nx\}+1/2}{n^{\beta}}.
\end{align*}

Let us first assume that $\beta<2$. While the first term is merely linear, it is easy to see that the second term is equivalent to $x^{\beta-1}/(2-\beta)$ as $x$ goes to zero. Concerning the third term, its absolute value may be bounded by the sum of $1/n^{\beta}$ over $n\geq\lceil 1/x\rceil$, which is equivalent to $x^{\beta-1}/(\beta-1)$. Given that $\beta>3/2$, the difference $|f_{\beta}(x)-f_{\beta}(0)|$ is thus of the order of $x^{\beta-1}$, and the result follows.

In the case where $\beta=2$, the second term in the above decomposition is equivalent to $x\log(1/x)$, and the upper bound on the third term is equivalent to $x$, so the result follows as well.

Finally, let us consider the case in which $\beta>2$. The upper bound on the third term above is again equivalent to $x^{\beta-1}/(\beta-1)$. Moreover, the second term may be rewritten as
\[
x\sum_{n=1}^{\lceil 1/x\rceil-1} \frac{1}{n^{\beta-1}}=\zeta(\beta-1)x-x\sum_{n=\lceil 1/x\rceil}^{\infty} \frac{1}{n^{\beta-1}},
\]
and the second term of the latter expression is equivalent to $x^{\beta-1}/(\beta-2)$. The result now follows because $\beta$ is not an integer.
\end{proof}

This example opens the possibility of studying the pointwise regularity of Davenport series for which the bound supplied by Lemma~\ref{lem:upbndjump} is not optimal. Beyond the study of particular functions at particular points, natural general open questions are the following. Under simple assumptions on the Davenport coefficients, can one show that the bound is optimal at a given point? everywhere? outside a set of dimension zero? or almost everywhere? Similar questions can also be raised in the more general setting of compensated pure jump functions. 

\subsection{Hecke's functions}

In dimension one, special attention has been paid to the study of very specific Davenport series, namely, Hecke's functions $\Hfrak_{\beta}$, which depend on a parameter $\beta\in\C$, and are defined by~(\ref{eq:defHecke}). Note that they can actually turn out to be distributions when the real part $\Re\beta$ is sufficiently small.

These functions have a rich history. They were first considered as functions of the complex variable $\beta$, the real number $x$ being merely a parameter. Hecke studied their analytic continuation, and his study was later extended by Hardy; these results showed that the range of analytic continuation depends on the Diophantine approximation properties of the parameter $x$. As a function of the real variable $x$, the spectrum of singularities was completely determined only in the case where $\Re\beta\geq 2$, which leaves open the case where $1<\Re\beta<2$, see~\cite{Jaffard:2004os} and~(\ref{eq:spechecke}) below. Note that the counterexample supplied in Section~\ref{subsec:conclud1} is closely related with Hecke's functions, and we refer to Section~\ref{subsec:conclud2} below for further connections. One could also consider multivariate extensions of these functions, specifically, the functions
\[
x\mapsto\sum_{n\in\ZD} \eps_n \frac{\{n\cdot x\}}{|n|^{\beta}},
\]
where $(\eps_{n})_{n\in\ZD}$ is an odd sequence taking the values $\pm 1$.

\subsection{Spectrum of singularities of compensated pure jumps functions}\label{subsec:conclud2}

We now go back to the general setting supplied by the compensated pure jump functions. The examples of such functions whose multifractal properties are known may be separated into two large classes. The first class corresponds to the case where the jump locations $x_{n}$ appearing in~(\ref{eq:cpjf}) are somehow homogeneously distributed; this is the case for many examples of Davenport series~\cite{Jaffard:2004os,Jaffard:2009ve,Jaffard:2010qf} or L\'evy fields and processes~\cite{Durand:2007fk,Durand:2010uq,Jaffard:1999fg}. The second class is composed of functions for which the jump locations $x_{n}$ form a strongly inhomogeneous sequence; such examples have been investigated by J.~Barral and S.~Seuret, and include L\'evy subordinators in multifractal time~\cite{Barral:2007ly} or heterogeneous sums of Dirac masses~\cite{Barral:2004fk,Barral:2005vn,Barral:2008zr}. In the heterogeneous case, the obtained spectra strongly differ from those which have been exhibited in the present paper. Indeed, they are usually composed of two parts: a linear one (for sufficiently small values of the H\"older exponent $h$), followed by a strictly concave one. Note however that spectra of a different kind have been obtained in~\cite{Barral:2010vn}; for some Markov processes which differ from L\'evy processes, one meets spectra that are a superposition of linear functions with different slopes.

In the homogeneous case, the spectra that have been met up to now are linear. However, this is not a general rule, even in the particular case of one-variable Davenport series. Let us consider for instance the function $f_{\beta}$ defined by~(\ref{eq:deffbeta}), and suppose that $\beta$ is a noninteger real number larger than two. The local spectrum of singularities of the corresponding Hecke function $\Hfrak_{\beta}$ defined by~(\ref{eq:defHecke}) is then supported by the interval $[0,\beta/2]$ and satisfies
\begin{equation}\label{eq:spechecke}
\forall h\in[0,\beta/2] \quad \forall W\neq\emptyset\mbox{ open} \qquad d_{\Hfrak_{\beta}}(h,W)=\frac{2h}{\beta}\,;
\end{equation}
this follows readily from the approach employed in~\cite{Jaffard:2004os} in order to compute the global spectrum of $\Hfrak_{\beta}$, which corresponds to the case where $W$ is equal to the whole real line. As shown by Proposition~\ref{prp:regfbetazero}, subtracting the term $\zeta(\beta)\{x\}$ to Hecke's function $\Hfrak_{\beta}(x)$ shifts the value of the H\"older exponent at the integers from zero to $\beta-1$. As a consequence, the local spectrum of singularities of the resulting function $f_{\beta}$ is now supported in the set $[0,\beta/2]\cup\{\beta-1\}$. Moreover, the function $f_{\beta}$ still satisfies~(\ref{eq:spechecke}) but, rather than being empty, its iso-H\"older set $E_{f_{\beta}}(\beta-1)$ is equal to $\Z$. Therefore, for any open subset $W$ of $\R$ that contains an integer,
\[
d_{f_{\beta}}(\beta-1,W)=0.
\]
In particular, the local spectrum of $f_{\beta}$ depends on the particular region that is considered. Thus, unlike the corresponding Hecke function $\Hfrak_{\beta}$, the function $f_{\beta}$ is not a homogeneous multifractal function. Furthermore, unlike that of $\Hfrak_{\beta}$, the global spectrum of singularities of $f_{\beta}$ is not a linear function. More specifically, the graph of this spectrum is the union of a segment and an isolated point. The multifractal properties of $f_{\beta}$ may be put in comparison with those of Riemann's function $\sum_{n} \sin(\pi n^{2}x)/n^{2}$ whose global spectrum of singularities has exactly the same shape and which is a homogeneous multifractal function, see~\cite{Jaffard:1996fk}.

This example raises several questions. Is there a simple condition on the coefficients of a Davenport series which ensures that its spectrum is linear? What is the general form of the spectrum of a Davenport series? Similar questions can also be asked in the more general setting of compensated pure jumps functions. Note that some results on these problems have been obtained by J.~Barral and S.~Seuret in the slightly different setting supplied by the {\em large deviation spectrum}, see~\cite{Barral:2007ve}.

\subsection{$p$-exponent}\label{subsec:conclud3} 

H\"older pointwise regularity is defined only for locally bounded functions, which explains why we always made the assumption that the sequence of Davenport coefficients belongs to $\ell^{1}$ when studying H\"older regularity. However, Proposition~\ref{prp:convpartialsumsbis} shows that, even when the sequence of Davenport coefficients does not belong to $\ell^{1}$, and therefore convergence in $L^{\infty}$ is no more guaranteed, one can obtain convergence in $L^{2}$ (this corresponds to the Sobolev space $H^{s}$ considered in Section~\ref{sec:sobol} in the case where $s$ is zero), and also for larger values of $p$; indeed when $s$ is positive, the Sobolev embeddings imply that, if $f$ belongs to $H^{s}$, then it also belongs to $L^{p}$ for all $p$ smaller than the critical value $p_{0}$ defined by the condition
\[
\frac{1}{p_{0}}=\frac{1}{2}-\frac{s}{d}.
\]
Note also that the specific case of $L^{2}$ convergence of one-variable Davenport series has already been considered, see~\cite{Bremont:2011fk,Jaffard:2004os}.

In such situations, one can still perform a pointwise analysis of regularity, by using a definition of pointwise smoothness which is weaker than H\"older regularity and is compatible with functions that are not locally bounded: It is the notion of $T^{p}_{\alpha}(x_{0})$ regularity, which was introduced by Calder\'on and Zygmund in~1961, see~\cite{Calderon:1961kx}. The next definition is an adaptation of Definition~\ref{df:hol} to that setting.

\begin{df}
Let $f$ be a tempered distribution on $\R^{d}$, let $p\in[1,\infty)$, let $\alpha>-d/p$ and let $x_{0}\in\R^{d}$. The distribution $f$ belongs to $T^{p}_{\alpha}(x_{0})$ if it coincides with an $L^{p}$ function in the open ball $\opball{x_{0}}{R}$ for some real $R>0$, and if there exist a real $C>0$ and a polynomial $P_{x_{0}}$ of degree less than $\alpha$ such that for all $r\in(0,R]$,
\[
\left( \frac{1}{r^{d}} \int_{\opball{x_{0}}{r}} |f(x)-P_{x_{0}}(x)|^{p} \,\dd x \right)^{1/p}
\leq C r^{\alpha}.
\]
The {\em $p$-exponent} of $f$ at $x_{0}$ is then defined as
\[
h^{p}_{f}(x_{0})=\sup\{ \alpha>-d/p \:|\: f \in T^{p}_{\alpha}(x_{0}) \}.
\]
\end{df}

Note that the H\"older exponent corresponds to the case where $p=\infty$, and the condition on the degree of $P_{x_{0}}$ implies its uniqueness. This definition is a natural substitute for pointwise H\"older regularity when functions in $L^{p}_{\rm loc}$ are considered. In particular, the $p$-exponent can take values down to $-d/p$, thereby allowing to take into account behaviors which are locally of the form $1/|x-x_{0}|^{\alpha}$ for $\alpha<d/p$.

Furthermore, similarly to~(\ref{eq:dfisohold}) and~(\ref{eq:dflocspec}), we may define the analogs of the iso-H\"older sets and the local spectrum of singularities by
\[
E^{p}_{f}(h)=\{x\in\R^{d}\:|\:h^{p}_{f}(x)=h\}
\qquad\mbox{and}\qquad
d^{p}_{f}(h,W)=\Hdim(E^{p}_{f}(h)\cap W),
\]
the latter quantity being referred to as the local $p$-spectrum of the distribution $f$.

Both in the univariate and the multivariate case, the subject of determining the $p$-exponents and the $p$-spectrum of a Davenport series with coefficients not belonging to $\ell^{1}$ is completely open.

\subsection{Directional regularity}

The notion of H\"older pointwise regularity given in Definition~\ref{df:hol} does not take into account directional regularity but yields the worst possible regularity in all directions. Therefore, all the results obtained in the present paper do not take into account possible directional irregularity phenomena.

We now briefly discuss the notion of directional regularity. Let $f$ be a locally bounded function defined on $\R^{d}$. In order to take into account directional behaviors, it is natural to define the H\"older regularity at $x_0$ in a direction $u\in\R^{d}\setminus\{0\}$ as the H\"older regularity at zero of the univariate function $t\mapsto f(x_{0}+t u)$. This definition has several drawbacks which stem from the fact that the latter function is defined as the trace of $f$ on a line, which is a set of measure zero, see~\cite{Jaffard:2010ys} for a detailed discussion. Let us now give the definition of anisotropic smoothness which is currently used, see {\em e.g.}~\cite{Ben-Slimane:1998uq,Jaffard:2010ys}.

\begin{df}
Let $f$ be a real-valued function defined on $\R^{d}$ and bounded in a neighborhood of a point $x_{0}\in\R^{d}$. Let $e=(e_{1},\ldots,e_{d})$ be an orthonormal basis of $\R^d$ and let $\alpha=(\alpha_{1},\ldots,\alpha_{d})$ be a $d$-tuple of nonnegative real numbers such that $\alpha_{1}\geq\ldots\geq\alpha_{d}$. The function $f$ belongs to $C^{\alpha}(x_{0},e)$ if there exist a real $C >0$ and a polynomial $P_{x_{0}}$ such that for all $x$ in a neighborhood of $x_{0}$,
\[
|f(x)-P_{x_{0}}(x)|\leq C\,\sum_{i=1}^{d}|(x-x_{0})\cdot e_{i}|^{\alpha_{i}}.
\]
\end{df}

Since, by construction, Davenport series display jumps along hyperplanes, thereby being extremely anisotropic by nature, a natural question is to determine their pointwise anisotropic regularity; the same remark is also relevant for L\'evy fields, which present the same type of anisotropy, see~\cite{Durand:2010uq}. These two examples would certainly be natural candidates to test possible definitions of {\em anisotropic spectra of singularities}.

Finally, note that an extension of pointwise smoothness combining anisotropy and the $T^{p}_{\alpha}(x_{0})$ condition is proposed in~\cite{Jaffard:2010ys}. This notion could be relevant in order to perform the study of the anisotropy of multivariate Davenport series with coefficients not belonging to $\ell^{1}$.

\section{Proof of Theorem~\ref{thm:formhold}}\label{sec:proofthmformhold}

Let us begin by comparing the two lower limits appearing in the statement of the theorem. First, given that $\theta_{a}$ is bounded above by one, we have
\[
\liminf_{q\to\infty\atop q\in\supp{A}}\frac{\log|A_{q}|}{\log\delta^{\Pcal}_{q}(x_{0})}
\leq\liminf_{q\to\infty\atop q\in\supp{\overline{a}}}\frac{\log|\overline{a}_{q}|}{\log\delta^{\Pcal}_{q}(x_{0})},
\]
where $A=J(a)$ and $\overline{a}=M(a)$. In addition, replacing $\delta^{\Pcal}_{q}(x_{0})$ by $\delta_{q}(x_{0})$ in the right-hand side does not change the value of the lower limit. Indeed, for any point $q$ in the support of the sequence $\overline{a}$, the distance $\delta_{q}(x_{0})$ is reached on a hyperplane of the form $H_{p,q}$ with $p\in\Z$, so that
\[
\delta_{q}(x_{0})=\frac{|q\cdot x_{0}-p|}{|q|}=\frac{|q'\cdot x_{0}-p'|}{|q'|}\geq\delta^{\Pcal}_{q'}(x_{0}).
\]
Here, $p'=p/r$ and $q'=q/r$, where $r$ denotes the greatest common divisor of the integer $p$ and the components of the vector $q$, so that $p'\in\prim{q'}$. Since the multiples of $q$ are also multiples of $q'$, the supremum over all integers $l\geq 1$ of $|a_{lq}|$ is at most that of $|a_{lq'}|$. As a consequence, for all $q\in\supp{\overline{a}}$ large enough, there exists a point $q'|q$ such that
\[
\frac{\log\overline{a}_{q'}}{\log\delta^{\Pcal}_{q'}(x_{0})}
\leq\frac{\log\overline{a}_{q}}{\log\delta_{q}(x_{0})}
\leq\frac{\log\overline{a}_{q}}{\log\delta^{\Pcal}_{q}(x_{0})},
\]
so the lower limit featuring $\delta^{\Pcal}_{q}(x_{0})$ in the denominator coincides with that featuring $\delta_{q}(x_{0})$. Furthermore, $|a_{q}|$ is obviously bounded above by $\overline{a}_{q}$, so that
\[
\frac{\log\overline{a}_{q}}{\log\delta_{q}(x_{0})}\leq\frac{\log|a_{q}|}{\log\delta_{q}(x_{0})}.
\]
The above discussion, combined with Corollary~\ref{cor:upbndjumpDav}, finally leads to~(\ref{eq:bndholdAa}). In particular, since $\delta_{n}(x_{0})$ is bounded above by $1/|n|$ regardless of the value of $x_{0}$, we deduce the next uniform bound on the H\"older exponent:
\[
h_{f}(x_{0})\leq\gamma_{a},
\]
where $\gamma_{a}$ is defined by~(\ref{eq:defgamma}). The remainder of the theorem then follows when the sequence of Davenport coefficients has slow decay, {\em i.e.}~when $\gamma_{a}$ vanishes.

In order to finish the proof of the theorem, we thus may assume that $\gamma_{a}$ is positive. It remains us to establish that
\begin{equation}\label{eq:geqloganlogdn}
h_{f}(x_{0})\geq\liminf_{n\to\infty \atop n\in\supp{a}}\frac{\log|a_{n}|}{\log\delta_{n}(x_{0})}
\end{equation}
for any point $x_{0}$ that does not belong to the set $\disc{M(a)}$. Let us consider an integer $j_{0}\geq 0$ and a point $x\in\R^{d}$ such that $2^{-(j_{0}+1)}\leq|x-x_{0}|<2^{-j_{0}}$. Since the sequence $a$ is in $\ell^{1}$, we may define
\[
\Sigma_{x_{0},x}(Z)=\sum_{n\in Z} a_{n}\left(\{n\cdot x\}-\{n\cdot x_{0}\}\right)
\]
for any subset $Z$ of $\Z^{d}$. Since the Davenport series converges normally, it is clear that its increment between $x_{0}$ and $x$ may be written in the form
\[
f(x)-f(x_{0})=\Sigma_{x_{0},x}(\Z^{d}).
\]
Therefore, it suffices to handle the series $\Sigma_{x_{0},x}(Z)$ for $Z$ ranging over a collection of sets that form a partition of $\Z^{d}$.

To be specific, let us begin by giving an upper bound on $|\Sigma_{x_{0},x}(\Ccal_{j}\cap\Z^{d})|$, where $\Ccal_{j}$ is the domain defined by
\[
\Ccal_{j}=\left\{x\in\R^{d} \:|\: 2^{j}\leq|x|<2^{j+1}\right\},
\]
for any integer $j\geq 0$. For every fixed $\gamma\in(0,\gamma_{a})$, there exists a constant $C_{\gamma}>0$ such that $|a_{n}|\leq C_{\gamma}|n|^{-\gamma}$ for all $n\in\Z^{d}$. For the sake of simplicity, we merely write $|a_{n}|\ll |n|^{-\gamma}$ in such a situation, thereby making use of the Vinogradov symbol. Accordingly,
\[
\left|\Sigma_{x_{0},x}(\Ccal_{j}\cap\Z^{d})\right|
\leq\sum_{n\in\Ccal_{j}\cap\Z^{d}} |a_{n}|
\ll\sum_{n\in\Ccal_{j}\cap\Z^{d}} |n|^{-\gamma}\ind_{\{n\in\supp{a}\}}.
\]
Furthermore, for every fixed $\eps>0$, we deduce from the sparsity assumption bearing on the sequence $a$ that $\card{(\supp{a}\cap\Ccal_{j})}\ll 2^{\eps j}$. As a consequence,
\[
\sum_{n\in\Ccal_{j}\cap\Z^{d}} |n|^{-\gamma}\ind_{\{n\in\supp{a}\}}\leq 2^{-\gamma j}\card{(\supp{a}\cap\Ccal_{j})}\ll 2^{(-\gamma+\eps)j}.
\]
The union over all integers $j\geq j_{0}$ is equal to the complement in $\R^{d}$ of the open ball centered at the origin with radius $2^{j_{0}}$. Thus, summing over all these values of $j$, we obtain
\[
\left|\Sigma_{x_{0},x}(\Z^{d}\setminus\opball{0}{2^{j_{0}}})\right|
\ll 2^{(-\gamma+\eps)j_{0}}
\ll |x-x_{0}|^{\gamma-\eps}.
\]

Now, let $\Ncal_{x_{0},x}$ denote the set of all points $n\in\Z^{d}$ for which there exists an integer $k\in\Z$ satisfying either $n\cdot x_{0}\leq k<n\cdot x$ or $n\cdot x\leq k<n\cdot x_{0}$, meaning that the hyperplane $H_{k,n}$ separates the points $x_{0}$ and $x$, and let $\Ncal^{\rm c}_{x_{0},x}$ denote its complement in $\Z^{d}$. For each $n\in\Ncal^{\rm c}_{x_{0},x}$, it is clear that $n\cdot x_{0}$ and $n\cdot x$ have the same integer part, so that
\begin{equation}\label{eq:SigmaNcalc}
\Sigma_{x_{0},x}(\Ncal^{\rm c}_{x_{0},x}\cap\opball{0}{2^{j_{0}}})=(x-x_{0})\cdot\sum_{n\in\Ncal^{\rm c}_{x_{0},x} \atop |n|<2^{j_{0}}} a_{n}n.
\end{equation}
If $\gamma_{a}$ is smaller than or equal to one, the modulus of the sum in~(\ref{eq:SigmaNcalc}) may be bounded above by the sum over $j\in\{0,\ldots,j_{0}-1\}$ of
\begin{equation}\label{eq:bndsumCj}
\begin{split}
\sum_{n\in\Ccal_{j}\cap\Z^{d}} |a_{n}n|
&\ll\sum_{n\in\Ccal_{j}\cap\Z^{d}} |n|^{1-\gamma}\ind_{\{n\in\supp{a}\}}\\
&\ll 2^{(1-\gamma)j}\card{(\supp{a}\cap\Ccal_{j})}\ll 2^{(1-\gamma+\eps)j},
\end{split}
\end{equation}
which entails that
\[
\left|\Sigma_{x_{0},x}(\Ncal^{\rm c}_{x_{0},x}\cap\opball{0}{2^{j_{0}}})\right|
\ll |x-x_{0}|\,2^{(1-\gamma+\eps)j_{0}}
\leq |x-x_{0}|^{\gamma-\eps}.
\]
In the opposite case where $\gamma_{a}$ is larger than one, we assume that $1<\gamma<\gamma_{a}$ and rewrite~(\ref{eq:SigmaNcalc}) as a difference of two terms in the following form:
\[
\Sigma_{x_{0},x}(\Ncal^{\rm c}_{x_{0},x}\cap\opball{0}{2^{j_{0}}})
=(x-x_{0})\cdot\sum_{n\in\Ncal^{\rm c}_{x_{0},x}} a_{n}n
-(x-x_{0})\cdot\sum_{n\in\Ncal^{\rm c}_{x_{0},x} \atop |n|\geq 2^{j_{0}}} a_{n}n.
\]
The first series is normally convergent because, in view of~(\ref{eq:bndsumCj}), we have
\[
\sum_{n\in\Ncal^{\rm c}_{x_{0},x}} |a_{n}n|
\leq\sum_{j=0}^{\infty}\sum_{n\in\Ccal_{j}\cap\Z^{d}} |a_{n}n|
\ll\sum_{j=0}^{\infty} 2^{(1-\gamma+\eps)j}<\infty.
\]
In order to handle the second term, we make use of~(\ref{eq:bndsumCj}) again; this implies that
\begin{equation}\label{eq:bndNcx0x}
\sum_{n\in\Ncal^{\rm c}_{x_{0},x} \atop |n|\geq 2^{j_{0}}} |a_{n}n|
\leq\sum_{j=j_{0}}^{\infty}\sum_{n\in\Ccal_{j}\cap\Z^{d}} |a_{n}n|
\ll\sum_{j=j_{0}}^{\infty} 2^{(1-\gamma+\eps)j}
\ll |x-x_{0}|^{-1+\gamma-\eps}.
\end{equation}
We finally deduce that
\[
\left|\Sigma_{x_{0},x}(\Ncal^{\rm c}_{x_{0},x}\cap\opball{0}{2^{j_{0}}})-(x-x_{0})\cdot\sum_{n\in\Ncal^{\rm c}_{x_{0},x}} a_{n}n\right|
\ll |x-x_{0}|^{\gamma-\eps}.
\]

It remains us to consider the behavior of $\Sigma_{x_{0},x}$ on the set $\Ncal_{x_{0},x}\cap\opball{0}{2^{j_{0}}}$. Let $\alpha$ denote the lower limit in the right-hand side of~(\ref{eq:geqloganlogdn}), which we may assume to be positive, and let $\alpha'\in(0,\alpha)$. By definition of $\alpha$, we have $\delta_{n}(x_{0})\geq|a_{n}|^{1/\alpha'}$ for $n\in\Z^{d}$ sufficiently far from the origin. In addition, $a_{n}$ necessarily vanishes when the distance $\delta_{n}(x_{0})$ is zero. As a matter of fact, in that situation, $x_{0}$ belongs to a hyperplane $H_{k,n}$ with $k\in\Z$. This hyperplane is represented by a unique pair $(p,q)\in\Hcal_{d}$, and then $p$ and $q$ divide $k$ and $n$, respectively. However, $x_{0}$ does not belong to $\disc{M(a)}$, so the index $q$ cannot belong to the support of the sequence $M(a)=\overline{a}$, which entails that $|a_{n}|\leq\overline{a}_{q}=0$. The upshot is that there exists a real $C_{\alpha'}>0$ such that $\delta_{n}(x_{0})\geq C_{\alpha'}|a_{n}|^{1/\alpha'}$ for all $n\in\Z^{d}$, which we write $|a_{n}|\ll\delta_{n}(x_{0})^{\alpha'}$ still using the Vinogradov symbol. As a consequence,
\[
\left|\Sigma_{x_{0},x}(\Ncal_{x_{0},x}\cap\opball{0}{2^{j_{0}}})\right|
\leq\sum_{n\in\Ncal_{x_{0},x} \atop |n|<2^{j_{0}}} |a_{n}|
\ll\sum_{n\in\Ncal_{x_{0},x} \atop |n|<2^{j_{0}}} \delta_{n}(x_{0})^{\alpha'}\ind_{\{n\in\supp{a}\}}.
\]
Moreover, when $n$ belongs to $\Ncal_{x_{0},x}$, we have $|n\cdot x_{0}-k|\leq|n\cdot(x-x_{0})|$ for some integer $k\in\Z$, so that $\delta_{n}(x_{0})$ is bounded above by $|x-x_{0}|$. Hence,
\[
\sum_{n\in\Ncal_{x_{0},x} \atop |n|<2^{j_{0}}} \delta_{n}(x_{0})^{\alpha'}\ind_{\{n\in\supp{a}\}}
\leq|x-x_{0}|^{\alpha'}\card{(\supp{a}\cap\opball{0}{2^{j_{0}}})}
\ll |x-x_{0}|^{\alpha'} 2^{\eps j_{0}},
\]
where the last bound follows from the sparsity assumption bearing on the sequence $a$. We deduce that
\[
\left|\Sigma_{x_{0},x}(\Ncal_{x_{0},x}\cap\opball{0}{2^{j_{0}}})\right|
\ll |x-x_{0}|^{\alpha'-\eps}.
\]
The above approach also enables us to write that
\[
\sum_{n\in\Ncal_{x_{0},x} \atop |n|<2^{j_{0}}} |a_{n}n|
\ll |x-x_{0}|^{\alpha'}\sum_{n\in\Ncal_{x_{0},x} \atop |n|<2^{j_{0}}} |n|\ind_{\{n\in\supp{a}\}}
\ll |x-x_{0}|^{\alpha'}\sum_{j=0}^{j_{0}-1} 2^{j}\card{(\supp{a}\cap\Ccal_{j})},
\]
where the last sum is bounded by $2^{(1+\eps)j_{0}}$ up to a constant, in view of the sparsity of the sequence $a$. In addition, when $1<\gamma<\gamma_{a}$, the bound given by~(\ref{eq:bndNcx0x}) still holds when $\Ncal^{\rm c}_{x_{0},x}$ is replaced by $\Ncal_{x_{0},x}$. It follows that
\[
\left|\sum_{n\in\Ncal_{x_{0},x}} a_{n}n\right|
\ll |x-x_{0}|^{-1+\gamma-\eps}+|x-x_{0}|^{-1+\alpha'-\eps},
\]
which readily implies that
\[
\left|\Sigma_{x_{0},x}(\Ncal_{x_{0},x}\cap\opball{0}{2^{j_{0}}})-(x-x_{0})\cdot\sum_{n\in\Ncal_{x_{0},x}} a_{n}n\right|
\ll |x-x_{0}|^{\gamma-\eps}+|x-x_{0}|^{\alpha'-\eps}.
\]

Combining all the previously obtained bounds, we finally get
\[
|f(x)-f(x_{0})|\ll |x-x_{0}|^{\gamma-\eps}+|x-x_{0}|^{\alpha'-\eps}
\]
when $\gamma_{a}$ is smaller than or equal to one, and
\[
\left|f(x)-f(x_{0})-(x-x_{0})\cdot\sum_{n\in\Z^{d}} a_{n}n\right|\ll |x-x_{0}|^{\gamma-\eps}+|x-x_{0}|^{\alpha'-\eps}
\]
when $\gamma_{a}$ is larger than one. In both cases, it appears that the H\"older exponent at $x_{0}$ of the Davenport series $f$ is at least the minimum between $\gamma-\eps$ and $\alpha'-\eps$. The bound~(\ref{eq:geqloganlogdn}) finally follows from letting $\eps$ go to zero, $\gamma$ to $\gamma_{a}$ and $\alpha'$ to $\alpha$.

\section{Proof of Theorems~\ref{thm:spec} and~\ref{thm:slisingsets}}\label{sec:proofthmmultifrac} 

Throughout the section, $f$ denotes a Davenport series with coefficients given by a sequence $a=(a_{n})_{n\in\ZD}$ in $\ell^{1}$. We assume that the series is sparse and not asymptotically jump canceling, and that $\gamma_{a}$ is both positive and finite.

\subsection{Locations of the singularities}

The first step to the proof of Theorems~\ref{thm:spec} and~\ref{thm:slisingsets} consists in observing that the iso-H\"older sets and the singularity sets of the Davenport series $f$ may be expressed in terms of the sets $L_{a}(\alpha)$ of all points that are at a distance less than $|a_{n}|^{1/\alpha}$ from a hyperplane $H_{k,n}$ defined as in~(\ref{eq:defHpq}) for infinitely many points $n$ in the support of the sequence $a$. Put another way, a point $x$ belongs to $L_{a}(\alpha)$ if and only if the distance $\delta_{n}(x)$ defined by~(\ref{eq:defdeltanx}) is less than $|a_{n}|^{1/\alpha}$ infinitely often. To be more specific, for any real $\alpha>0$, the set $L_{a}(\alpha)$ is defined by
\begin{equation}\label{eq:defLalpha}
L_{a}(\alpha)=\bigl\{x\in\R^{d}\:\bigl|\: |n\cdot x-k|<|n|\,|a_{n}|^{1/\alpha} \mbox{ for i.m.~} (k,n)\in\Z\times\Z^{d}\bigr\},
\end{equation}
where i.m.~stands for ``infinitely many''. It is easy and useful to remark that the mapping $\alpha\mapsto L_{a}(\alpha)$ is nondecreasing.

The connexion between the iso-H\"older and singularity sets, and the sets $L_{a}(\alpha)$ is now given by the next result. It is a direct consequence of Theorem~\ref{thm:formhold}, along with the discussion made in Section~\ref{sec:disc} above according to which the Davenport series $f$ is discontinuous on $\disc{J(a)}$, thus having H\"older exponent zero thereon. In its statement, $M(a)$ denotes the image of the sequence $a$ under the action of the maximal operator.

\begin{lem}\label{lem:EhLalpha}
Let $h\in[0,\gamma_{a}]$. Then,
\begin{equation}\label{eq:lem:EhLalpha1}
E_{f}(h)\setminus\disc{J(a)}\subseteq E'_{f}(h)\subseteq\left(\disc{M(a)}\setminus\disc{J(a)}\right)\cup\bigcap_{\alpha>h} L_{a}(\alpha).
\end{equation}
Moreover, $E_{f}(0)\supseteq\disc{J(a)}$ and for the positive values of $h$,
\begin{equation}\label{eq:lem:EhLalpha2}
E'_{f}(h)\supseteq L_{a}(h)\setminus\disc{J(a)} \qquad\mbox{and}\qquad E_{f}(h)\supseteq E'_{f}(h)\setminus\bigcup_{\alpha<h}L_{a}(\alpha).
\end{equation}
\end{lem}

Lemma~\ref{lem:EhLalpha} suggests that the proof of Theorems~\ref{thm:spec} and~\ref{thm:slisingsets} will follow from a detailed understanding of the size and large intersection properties of the sets $L_{a}(\alpha)$. This is the purpose of the next subsection, but let us just point out here that
\begin{equation}\label{eq:LalphaRd}
\bigcap_{\alpha>\gamma_{a}} L_{a}(\alpha)=\R^{d}.
\end{equation}
Indeed, it is plain that $\delta_{n}(x)\leq 1/|n|$ for every $n\in\ZD$ and every point $x\in\R^{d}$, and that $|a_{n}|^{1/\alpha}|n|\geq 1$ infinitely often, when $\alpha>\gamma_{a}$. We may thus restrict our attention to the case where $\alpha\leq\gamma_{a}$ in the study of the size and large intersection properties of $L_{a}(\alpha)$.

\subsection{Size and large intersection properties of the sets $L_{a}(\alpha)$, connection with the Duffin-Schaeffer and Catlin conjectures}\label{subsec:sliLaalpha}

We shall investigate the size properties of the sets $L_{a}(\alpha)$ by estimating their Hausdorff measures for specific gauge functions. We call a gauge function any continuous nondecreasing function $g$ which is defined on $[0,\eps]$ for some $\eps>0$ and vanishes at zero. The Hausdorff measure associated with such a gauge function is then defined by
\[
\hau^{g}(E)=\lim_{\delta\downarrow 0}\uparrow \hau^{g}_{\delta}(E) \qquad\mbox{with}\qquad \hau^{g}_{\delta}(E)=\inf_{E\subseteq\bigcup_{i} U_{i}\atop\diam{U_{i}}<\delta} \sum_{i=1}^{\infty} g(\diam{U_{i}}),
\]
for any subset $E$ of $\R^{d}$. Here, the infimum is taken over all sequences $(U_{i})_{i\geq 1}$ of subsets of $\R^{d}$ satisfying $E\subseteq\bigcup_{i}U_{i}$ and $\diam{U_{i}}<\delta$ for all $i$, were $\diam{\,\cdot\,}$ denotes diameter. It is well-known that $\hau^{g}$ is a Borel measure on $\R^{d}$, see {\em e.g.}~\cite{Rogers:1970wb}. Moreover, the Hausdorff measure associated with the gauge function $r\mapsto r^{s}$ is called the $s$-dimensional Hausdorff measure and is denoted by $\hau^{s}$; recall that such measures enable one to define the Hausdorff dimension of a nonempty set $E\subseteq\R^{d}$ by
\[
\Hdim E=\sup\{ s\in(0,d) \:|\: \hau^{s}(E)=\infty \}=\inf\{ s\in(0,d) \:|\: \hau^{s}(E)=0 \},
\]
see Falconer's book~\cite{Falconer:2003oj} for instance.

Theorems~\ref{thm:spec} and~\ref{thm:slisingsets} state that the iso-H\"older and the singularity sets of the Davenport series $f$ all have Hausdorff dimension between $d-1$ and $d$. Therefore, on our way to the proof of these results, we may restrict our attention to the gauge functions of the form $r\mapsto r^{d-1+s}$, with $0\leq s\leq 1$, as well as slight corrections thereof. These corrections are obtained by replacing $r^{s}$ in the previous expression by more general functions, specifically, the continuous nondecreasing functions $\ph$ defined on $[0,\eps]$ for some $\eps>0$ which vanish at zero, for which $r\mapsto\ph(r)/r$ is nonincreasing and positive, and for which the limit
\[
s_{\ph}=\lim_{r\to 0}\frac{\log\ph(r)}{\log r}
\]
exists (this limit is then between zero and one). The collection of all such functions is denoted by $\Phi$, and clearly contains the functions $r\mapsto r^{s}$, for $0\leq s\leq 1$. For any $\ph\in\Phi$, it is now plain that the function $r\mapsto r^{d-1}\ph(r)$ is a gauge; the corresponding Hausdorff measure is denoted by $\hau^{d-1,\ph}$, and the value that it assigns to the set $L_{a}(\alpha)$ is discussed in the next statement.

\begin{lem}\label{lem:CShauzero}
For any real $\alpha>0$ and any function $\ph\in\Phi$,
\[
\sum_{n\in\Z^{d}} |n|\ph(|a_{n}|^{1/\alpha})<\infty \qquad\Longrightarrow\qquad \hau^{d-1,\ph}(L_{a}(\alpha))=0.
\]
\end{lem}

\begin{proof}
Let $\rho_{n}=|a_{n}|^{1/\alpha}$ for any $n\in\Z^{d}$, and let us consider two real numbers $A>1$ and $\delta\in(0,1]$. Let us assume that the series appearing in the statement of the lemma converges. So, there necessarily exists an integer $\eta_{0}\geq 1$ such that $4\rho_{n}<\delta$ for all $n\in\Z^{d}$ with $|n|\geq\eta_{0}$. Then, for any $\eta_{1}\geq\eta_{0}$,
\[
L_{a}(\alpha)\cap\opball{0}{A-1}\subseteq\bigcup_{n\in\supp{a} \atop |n|\geq\eta_{1}}\bigcup_{k\in\Z\atop |k|<A|n|}\left\{ x\in\opball{0}{A} \:\bigl|\: \dist(x,H_{k,n})<\rho_{n} \right\}.
\]
Moreover, each set in the union above may be covered by $(2\lfloor 2A\sqrt{d}/\rho_{n}\rfloor)^{d-1}$ open balls with radius $2\rho_{n}$. Therefore,
\begin{eqnarray*}
\hau^{r\mapsto r^{d-1}\ph(r)}_{\delta}(L_{a}(\alpha)\cap\opball{0}{A-1})
&\leq& \sum_{n\in\supp{a} \atop |n|\geq\eta_{1}} 2A|n|\left(\frac{4A\sqrt{d}}{\rho_{n}}\right)^{d-1} (4\rho_{n})^{d-1} \ph(4\rho_{n}) \\
&\leq& 8A(16A\sqrt{d})^{d-1} \sum_{n\in\Z^{d} \atop |n|\geq\eta_{1}} |n|\ph(|a_{n}|^{1/\alpha}).
\end{eqnarray*}
Letting $\eta_{1}\to\infty$ and $\delta\to 0$, we deduce that $\hau^{d-1,\ph}(L_{a}(\alpha)\cap\opball{0}{A-1})=0$. This holds for all integers $A\geq 1$, so the result follows.
\end{proof}

In particular, letting $\ph$ be the identity function in the statement of Lemma~\ref{lem:CShauzero} and letting $\leb^{d}$ be the Lebesgue measure in $\R^{d}$, we deduce that
\begin{equation}\label{eq:LalphaLebzero}
\sum_{n\in\Z^{d}} |n|\,|a_{n}|^{1/\alpha}<\infty \qquad\Longrightarrow\qquad \leb^{d}(L_{a}(\alpha))=0.
\end{equation}
Owing to the alternate expression~(\ref{eq:defgammabis}) of $\gamma_{a}$, it is easy to see that the above series converges once $\alpha$ is less than $\gamma_{a}$. Owing to Lemma~\ref{lem:EhLalpha} and the fact that $\disc{M(a)}$ is a countable union of hyperplanes, we deduce that the iso-H\"older sets $E_{f}(h)$ and the singularity sets $E'_{f}(h)$ have Lebesgue measure zero when $h<\gamma_{a}$. Together with Corollary~\ref{cor:bndgamma}, this implies that the sets $E_{f}(\gamma_{a})$ and $E'_{f}(\gamma_{a})$ both have full Lebesgue measure. Therefore, $h_{f}(x_{0})=\gamma_{a}$ for $\leb^{d}$-almost every $x_{0}\in\R^{d}$.

To proceed with the proof of Theorems~\ref{thm:spec} and~\ref{thm:slisingsets}, we shall need a kind of converse to Lemma~\ref{lem:CShauzero}, which ensures that $L_{a}(\alpha)$ has a positive Hausdorff measure for specific gauge functions. The study of the size properties of various classical sets arising in the metric theory of Diophantine approximation suggests that such a converse should look like the following:
\begin{equation}\label{eq:conjdivLalpha}
\sum_{n\in\Z^{d}} |n|\ph(|a_{n}|^{1/\alpha})=\infty \ \Longrightarrow\  \forall W\mbox{ open}\quad \hau^{d-1,\ph}(L_{a}(\alpha)\cap W)=\hau^{d-1,\ph}(W),
\end{equation}
see for instance~\cite{Beresnevich:2009vn,Durand:2007uq} and references therein. Given that $L_{a}(\alpha)$ is of the form
\[
K_{d}(\psi)=\bigl\{x\in\R^{d}\:\bigl|\: |n\cdot x-k|<\psi(n) \mbox{ for i.m.~} (k,n)\in\Z\times\Z^{d}\bigr\},
\]
where $\psi:\Z^{d}\to [0,\infty)$ is a multivariate approximating function, and thanks to the mass transference principle of~\cite{Beresnevich:2005vn} and the slicing technique of~\cite{Beresnevich:2006lr}, this would follow from the next general statement: The set $K_{d}(\psi)$ has full Lebesgue measure in $\R^{d}$ if the series $\sum_{n}\psi(n)$ diverges. However, such a statement is known to be false and we refer to Section~5 in~\cite{Beresnevich:2009vn} for a counterexample. The expected result is actually given by a generalization of the Catlin conjecture to the case of dual approximation, which has been formulated by V.~Beresnevich, V.~Bernik, M.~Dodson and S.~Velani~\cite{Beresnevich:2009vn}, and consists in replacing the previous series by
\[
\sum_{q\in\ZD} \phi_{d}(q)\max_{t\geq 1}\frac{\psi(t q)}{t|q|_{\infty}},
\]
where $|\cdot|_{\infty}$ denotes the supremum norm and $\phi_{d}(q)$ is the number of positive integers less than or equal to $|q|_{\infty}$ which are coprime with the components of $q$. We refer to~\cite{Beresnevich:2009vn} for a motivation of this conjecture, and for its relationship with the dual form of the famous Duffin-Schaeffer conjecture. The upshot is that it seems rather difficult to provide a converse to Lemma~\ref{lem:CShauzero} in the form~(\ref{eq:conjdivLalpha}).

As shown by the statement of Theorems~\ref{thm:spec} and~\ref{thm:slisingsets}, we ultimately describe the size of the iso-H\"older and singularity sets in terms of Hausdorff dimension, rather than using general Hausdorff measures. Thus, we do not need to call upon such precise results as those mentioned just above. In fact, we only need to prove that appropriate modifications of the set
\[
L_{a}(\gamma_{a})=\bigl\{x\in\R^{d}\:\bigl|\: |n\cdot x-k|<|n|\,|a_{n}|^{1/\gamma_{a}} \mbox{ for i.m.~} (k,n)\in\Z\times\Z^{d}\bigr\},
\]
obtained by letting $\alpha=\gamma_{a}$ in~(\ref{eq:defLalpha}), have full Lebesgue measure in $\R^{d}$. Note that the set $L_{a}(\alpha)$ has full Lebesgue measure in $\R^{d}$ when $\alpha>\gamma_{a}$ due to~(\ref{eq:LalphaRd}), and Lebesgue measure zero when $\alpha<\gamma_{a}$ by virtue of Lemma~\ref{lem:CShauzero}. The study of its Lebesgue measure for the critical value $\alpha=\gamma_{a}$ is more delicate. Indeed, if true,~(\ref{eq:conjdivLalpha}) would imply that $L_{a}(\gamma_{a})$ has full Lebesgue measure when the series $\sum_{n} |n|\,|a_{n}|^{1/\gamma_{a}}$ diverges. However, the coefficients of the Davenport $f$ series may be chosen in such a way that the series converges, {\em e.g.}~when the nonvanishing coefficients are given by $|a_{\lambda_{m}}|=(m^{2}|\lambda_{m}|)^{-\gamma}$ for some positive real $\gamma$ and some sparse injective sequence $(\lambda_{m})_{m\geq 1}$, in which case $L_{a}(\gamma_{a})$ has Lebesgue measure zero, by Lemma~\ref{lem:CShauzero}. This means that we shall have to slightly reshape this set in order to ensure that we work with a set with full Lebesgue measure. Actually, as shown by the next lemma, a slight modification of $L_{a}(\gamma_{a})$ enables one to recover the whole space. This modification is written in the form
\[
L_{a}^{(\ph,i)}(\gamma_{a})=\bigl\{x\in\R^{d}\:\bigl|\: |n\cdot x-k|<|n|\,\ph(|a_{n}|^{1/\gamma_{a}}) \mbox{ for i.m.~} (k,n)\in\Z\times\Ncal_{i}\bigr\},
\]
where $\ph$ and $i$ are appropriately chosen in $\Phi$ and $\{1,\ldots,d\}$ respectively. Here, $\Ncal_{i}$ denotes the set of all points $n=(n_{1},\ldots,n_{d})$ in $\Z^{d}$ such that $|n|_{\infty}=|n_{i}|$; note that the sets $\Ncal_{i}$ obviously form a covering of $\Z^{d}$. Whereas the function $\ph$ is crucial in order to enlarge the set $L_{a}(\gamma_{a})$ and then to recover the whole space, the index $i$, which is used to retain only some specific frequencies among the support of $a$, is introduced merely for technical reasons appearing in the proof of Lemma~\ref{lem:lipLalpha} below. In the following, $\Phi_{\star}$ denotes the collection of all functions $\ph\in\Phi$ with $s_{\ph}=1$ for which at least one of the sets $L_{a}^{(\ph,1)}(\gamma_{a}),\ldots,L_{a}^{(\ph,d)}(\gamma_{a})$ has full Lebesgue measure in $\R^{d}$. The next lemma shows that $\Phi_{\star}$ is nonempty.

\begin{lem}\label{lem:existphistar}
There exist an index $i_{\star}$ and a function $\ph_{\star}$ with $s_{\ph_{\star}}=1$ such that
\[
|n|\,\ph_{\star}(|a_{n}|^{1/\gamma_{a}})\geq 1 \qquad \mbox{for i.m.} \quad n\in\Ncal_{i_{\star}}.
\]
In particular, the set $L_{a}^{(\ph_{\star},i_{\star})}(\gamma_{a})$ is equal to the whole space $\R^{d}$, and $\ph_{\star}$ is in $\Phi_{\star}$.
\end{lem}

\begin{proof}
Let $(\lambda_{m})_{m\geq 1}$ denote an enumeration of the support of the sequence $a$. Given that the Davenport series is sparse, the sequence $(\lambda_{m})_{m\geq 1}$ is sparse and injective and, up to rearranging its terms, we may assume that the sequence $(|\lambda_{m}|)_{m\geq 1}$ is nondecreasing. Then, let $\rho_{m}=|a_{\lambda_{m}}|^{1/\gamma_{a}}$ and $u_{m}=1/|\lambda_{m}|$ for any integer $m\geq 1$. The case in which the sequence $(u_{m}/\rho_{m})_{m\geq 1}$ does not diverge to infinity is elementary. Indeed, in that situation, there exists a constant $C>0$ such that $u_{m}\leq C\rho_{m}$ for all $m$ belonging to some infinite subset $\Mcal$ of $\N$. Then, the function defined by $\ph_{\star}(r)=Cr$ clearly satisfies the required properties, and it suffices to choose $i_{\star}$ in such a way that $\Ncal_{i_{\star}}$ contains infinitely many points $\lambda_{m}$ with $m\in\Mcal$.

We may therefore suppose from now on that $(u_{m}/\rho_{m})_{m\geq 1}$ diverges to infinity. The definition~(\ref{eq:defgamma}) of $\gamma_{a}$, combined with the observation that $\Z^{d}$ is covered by the sets $\Ncal_{i}$, ensures the existence of an index $i_{\star}$ such that
\[
\limsup_{m\to\infty \atop \lambda_{m}\in\Ncal_{i_{\star}}}\frac{\log u_{m}}{\log\rho_{m}}=1.
\]
In addition, both $u_{m}$ and $\rho_{m}$ tend to zero as $m\to\infty$. Thus, we may find a sequence of indices $(m_{k})_{k\geq 1}$ in $\N$ along which all the following properties hold: The sequence $(u_{m_{k}}/\rho_{m_{k}})_{k\geq 1}$ diverges to infinity monotonically; $\log u_{m_{k}}/\log\rho_{m_{k}}$ tends to one as $k$ goes to infinity; for all $k\geq 1$,
\[
\left\{\begin{array}{l}
\lambda_{m}\in\Ncal_{i_{\star}} \\
\rho_{m_{k+1}}\leq (\rho_{m_{k}})^{k} \\
u_{m_{k+1}}\leq (u_{m_{k}})^{k}.
\end{array}\right.
\]
It is now straightforward to check that any logarithmic interpolation of the points $(\rho_{m_{k}},u_{m_{k}})$ yields a suitable function $\ph_{\star}$. To be specific, any function $\ph_{\star}$ defined on $[0,\infty)$ for which
\[
\log\ph_{\star}(r)=\log u_{m_{k}}-\frac{\log\rho_{m_{k}}-\log r}{\log\rho_{m_{k}}-\log\rho_{m_{k+1}}}(\log u_{m_{k}}-\log u_{m_{k+1}}),
\]
for all $r\in(\rho_{m_{k+1}},\rho_{m_{k}}]$ and $k\geq 1$, clearly belongs to the set $\Phi$ and meets all our requirements.
\end{proof}

For any $\alpha\in(0,\gamma_{a}]$, we now define a mapping $T_{\alpha}$ on the set $\Phi_{\star}$ by letting
\[
T_{\alpha}\ph:r\mapsto\ph(r^{\alpha/\gamma_{a}})
\]
for any function $\ph\in\Phi_{\star}$. All the functions $T_{\alpha}\ph$ belong to $\Phi$ and satisfy $s_{T_{\alpha}\ph}=\alpha/\gamma_{a}$, so they roughly behave like $r^{\alpha/\gamma_{a}}$ near the origin. Furthermore, $T_{\gamma_{a}}$ is the identity mapping. Now, recall that Lemma~\ref{lem:existphistar} yields a function $\ph_{\star}$ for which $|n|\,\ph_{\star}(|a_{n}|^{1/\gamma_{a}})\geq 1$ infinitely often; the series $\sum_{n} |n|T_{\alpha}\ph_{\star}(|a_{n}|^{1/\alpha})$ thus diverges. We are in a situation where the assumption of Lemma~\ref{lem:CShauzero} fails and, in fact, the set $L_{a}(\alpha)$ does not necessarily have a vanishing $\hau^{d-1,T_{\alpha}\ph_{\star}}$-mass. On the contrary, it should be regarded as large and omnipresent in $\R^{d}$ in terms of $(d-1+\alpha/\gamma_{a})$-dimensional Hausdorff measure, in the sense that it belongs to Falconer's class $\licfalc{d-1+\alpha/\gamma_{a}}$ of sets with large intersection defined above.

This follows from our next result, namely, Lemma~\ref{lem:lipLalpha}, which describes the large intersection properties of the sets $L_{a}(\alpha)$. The properties are expressed by means of the classes $\lic{g}{W}$ that were introduced in~\cite{Durand:2007uq} in order to extend Falconer's classes to general gauge functions $g$ and open sets $W\subseteq\R^{d}$, with a view to establishing a suitable framework to describe precisely the large intersection properties of various sets arising in the metric theory of Diophantine approximation. In what follows, we shall restrict our attention to the classes $\lic{d-1,\ph}{\R^{d}}$ of sets with large intersection in the whole space $\R^{d}$ with respect to gauge functions of the form $r\mapsto r^{d-1}\ph(r)$ with $\ph\in\Phi$. We refer to~\cite{Durand:2007uq} for a precise definition of those classes and a description of their main properties, and we content ourselves here with recalling that the class $\lic{d-1,\ph}{\R^{d}}$ is closed under countable intersections and bi-Lipschitz mappings, and is formed of $G_{\delta}$-subsets $E$ of $\R^{d}$ satisfying
\begin{equation}\label{eq:sliinfhau}
\forall W\neq\emptyset\mbox{ open} \qquad \hau^{d-1,\psi}(E\cap W)=\infty
\end{equation}
for any function $\psi\in\Phi$ growing faster than $\ph$ at zero, in the sense that $\psi/\ph$ tends to infinity monotonically, in which case we write $\psi\prec\ph$. A straightforward consequence of these properties is the fact that, when $d-1+s_{\ph}$ is positive, the class $\lic{d-1,\ph}{\R^{d}}$ is included in Falconer's class $\licfalc{d-1+s_{\ph}}$. Let us now describe the large intersection properties of the sets $L_{a}(\alpha)$; in the next statement, $\Phi_{\alpha}$ denotes the collection of functions $\ph\in\Phi$ satisfying $\ph\prec T_{\alpha}\ph_{\star}$ for some $\ph_{\star}\in\Phi_{\star}$.

\begin{lem}\label{lem:lipLalpha}
For any real $\alpha\in(0,\gamma_{a}]$,
\[
L_{a}(\alpha)\in\bigcap_{\ph\in\Phi_{\alpha}}\lic{d-1,\ph}{\R^{d}}\subseteq\licfalc{d-1+\alpha/\gamma_{a}}.
\]
\end{lem}

\begin{proof}
Let us consider a function $\ph\in\Phi_{\alpha}$, and a function $\ph_{\star}\in\Phi_{\star}$ for which $\ph\prec T_{\alpha}\ph_{\star}$. We shall make use of a ubiquity result, which enables one to deduce the large intersection properties of the set $L_{a}(\alpha)$ from the sole fact that a corresponding enlarged set, namely, one of the sets $L_{a}^{(\ph_{\star},1)}(\gamma_{a}),\ldots,L_{a}^{(\ph_{\star},d)}(\gamma_{a})$ has full Lebesgue measure in $\R^{d}$. Indeed, there exists an index $i$ such that Lebesgue-almost every point $x\in\R^{d}$ belongs to $L_{a}^{(\ph_{\star},i)}(\gamma_{a})$, {\em i.e.}~satisfies
\[
\dist(x,H_{k,n})<T_{\alpha}\ph_{\star}(|a_{n}|^{1/\alpha}) \qquad\mbox{ for i.m.}\quad (k,n)\in\Z\times\Ncal_{i}.
\]
Moreover, letting $U_{i}$ denote the line spanned by the $i$-th vector of the canonical basis of $\R^{d}$, we see that the hyperplanes $H_{k,n}$ are such that
\[
\sup_{n\in\Ncal_{i} \atop k\in\Z}\diam{\{x\in U_{i}\:|\: \dist(x,H_{k,n})<1\}}<\infty.
\]
We may therefore apply Theorem~3.6 in~\cite{Durand:2008jk} and deduce that $L_{a}(\alpha)\in\lic{d-1,\ph}{\R^{d}}$. To finish the proof, it suffices to consider a function $\ph\in\Phi_{\alpha}$ with $s_{\ph}=\alpha/\gamma_{a}$, and to recall that $\lic{d-1,\ph}{\R^{d}}\subseteq\licfalc{d-1+s_{\ph}}$. Such functions exist; as a matter of fact, one may take $\ph(r)=T_{\alpha}\ph_{\star}(r)\log(1/T_{\alpha}\ph_{\star}(r))$ where $\ph_{\star}$ is given by Lemma~\ref{lem:existphistar}.
\end{proof}

The above results being established, we are now in position to prove Theorems~\ref{thm:spec} and~\ref{thm:slisingsets}. This is the purpose of the last part of this section.

\subsection{End of the proof}

Recall that we have to establish the following properties: The iso-H\"older sets $E_{f}(h)$ and the singularity sets $E'_{f}(h)$ of the Davenport series $f$ have Hausdorff dimension equal to $d-1+h/\gamma_{a}$ in every nonempty open set. We also need to show that the latter sets belong to the classes $\varlicfalc{d-1+h/\gamma_{a}}$ when $h$ is positive. By virtue of Theorem~D in~\cite{Falconer:1994hx}, this implies that the singularity sets have packing dimension equal to $d$ in every nonempty open set, a feature that is also mentioned in the statement of Theorem~\ref{thm:slisingsets}. In view of various remarks written above, it only remains to consider the case where $h<\gamma_{a}$ and to establish the following three propositions.

\begin{prp}
For any real number $h\in[0,\gamma_{a})$,
\[
\max\{\Hdim E_{f}(h),\Hdim E'_{f}(h)\}\leq d-1+\frac{h}{\gamma_{a}}.
\]
\end{prp}

\begin{proof}
We begin by making use of Lemma~\ref{lem:EhLalpha}. The two inclusions~(\ref{eq:lem:EhLalpha1}), combined with the fact that the sets $\disc{J(a)}$ and $\disc{M(a)}$ are countable unions of hyperplanes, imply that
\[
\left\{\begin{array}{l}
\Hdim E_{f}(h)\leq\max\{d-1,\Hdim E'_{f}(h)\} \\[2mm]
\Hdim E'_{f}(h)\leq\max\left\{d-1,\inf\limits_{\alpha>h}\Hdim L_{a}(\alpha)\right\}
\end{array}\right.
\]
To conclude, it suffices to apply Lemma~\ref{lem:CShauzero} which, along with the alternate expression~(\ref{eq:defgammabis}) of $\gamma_{a}$, ensures that the dimension of $L_{a}(\alpha)$ is bounded above by $d-1+\alpha/\gamma_{a}$.
\end{proof}

In the next statement, $\varlic{d-1,\ph}{\R^{d}}$ denotes the extended class of sets with large intersection that is defined in terms of the initial class $\lic{d-1,\ph}{\R^{d}}$ by the following condition: For all $E\subseteq\R^{d}$,
\[
E\in\varlic{d-1,\ph}{\R^{d}} \qquad\Longleftrightarrow\qquad \exists E'\in\lic{d-1,\ph}{\R^{d}} \quad E'\subseteq E.
\]
The purpose of this extension is to avoid checking that the singularity sets are $G_{\delta}$-sets, which is inessential here. It is easy to see that the extended class $\varlic{d-1,\ph}{\R^{d}}$ contains the initial class $\lic{d-1,\ph}{\R^{d}}$ and coincides with the latter on the $G_{\delta}$-sets, in view of~\cite[Proposition~1(e)]{Durand:2007uq}. Moreover, the extended class enjoys the same remarkable properties as the initial class: $\varlic{d-1,\ph}{\R^{d}}$ is closed under countable intersections and bi-Lipschitz mappings, and its members satisfy~(\ref{eq:sliinfhau}). Moreover, when $d-1+s_{\ph}$ is positive, the class $\varlic{d-1,\ph}{\R^{d}}$ is included in the corresponding extended version $\varlicfalc{d-1+s_{\ph}}$ of Falconer's class.

\begin{prp}\label{prp:slisingsets}
Let us consider a real number $h\in(0,\gamma_{a})$. Then,
\[
E'_{f}(h)\in\bigcap_{\ph\in\Phi_{h}}\varlic{d-1,\ph}{\R^{d}}\subseteq\varlicfalc{d-1+h/\gamma_{a}}.
\]
\end{prp}

\begin{proof}
Lemma~\ref{lem:lipLalpha} ensures that the set $L_{a}(h)$ belongs to the classes $\lic{d-1,\ph}{\R^{d}}$ associated with the functions $\ph\in\Phi_{h}$. The same property holds for the set $\R^{d}\setminus\disc{J(a)}$; indeed, being the complement of a countable union of hyperplanes, this set is a $G_{\delta}$-set with full Lebesgue measure in $\R^{d}$, and such a set belongs to all the classes $\lic{g}{\R^{d}}$, see~\cite[Proposition~11]{Durand:2007uq}. Using the stability under intersection of the classes of sets with large intersection, we deduce that
\[
L_{a}(h)\setminus\disc{J(a)}\in\bigcap_{\ph\in\Phi_{h}}\lic{d-1,\ph}{\R^{d}}\subseteq\licfalc{d-1+h/\gamma_{a}},
\]
where the last inclusion also follows from Lemma~\ref{lem:lipLalpha}. To conclude, it suffices to make use of Lemma~\ref{lem:EhLalpha}, which ensures that $L_{a}(h)\setminus\disc{J(a)}$ is a subset of $E'_{f}(h)$, see the first inclusion in~(\ref{eq:lem:EhLalpha2}).
\end{proof}

Our last statement gives a lower bound on the Hausdorff dimension of the sets $E_{f}(h)$ and $E'_{f}(h)$ in every nonempty open subset of $\R^{d}$. Recall that $E_{f}(0)$ contains the set $\disc{J(a)}$, by virtue of Lemma~\ref{lem:EhLalpha}. As $\gamma_{a}$ is finite, the latter set is a dense countable union of hyperplanes, thereby having Hausdorff dimension at least $d-1$ in every nonempty open set. We may therefore restrict our attention to the positive values of $h$.

\begin{prp}
Let us consider a real number $h\in(0,\gamma_{a})$ and a nonempty open subset $W$ of $\R^{d}$. Then,
\[
\min\{\Hdim (E_{f}(h)\cap W),\Hdim (E'_{f}(h)\cap W)\}\geq d-1+\frac{h}{\gamma_{a}}.
\]
\end{prp}

\begin{proof}
It follows from Proposition~\ref{prp:slisingsets} that the singularity set $E'_{f}(h)$ satisfies~(\ref{eq:sliinfhau}) for any function $\psi\in\Phi$ such that $\psi\prec\ph$ for some $\ph\in\Phi_{h}$. Moreover, choosing a function $\psi$ for which $s_{\psi}=h/\gamma_{a}$, we also deduce from Lemma~\ref{lem:CShauzero} that all the sets $L_{a}(\alpha)$, for $\alpha<h$, have a vanishing $\hau^{d-1,\psi}$-mass. Such a function $\psi$ exists: It suffices to take $\psi(r)=T_{h}\ph_{\star}(r)\left(\log(1/T_{h}\ph_{\star}(r))\right)^{2}$, and also $\ph(r)=T_{h}\ph_{\star}(r)\log(1/T_{h}\ph_{\star}(r))$, where $\ph_{\star}$ is given by Lemma~\ref{lem:existphistar}. We finally get
\[
\hau^{d-1,\psi}\left(E_{f}(h)\cap W\right)\geq\hau^{d-1,\psi}\left(E'_{f}(h)\cap W\right)=\infty,
\]
where the first inequality is due to the second inclusion in~(\ref{eq:lem:EhLalpha2}), which appears in the statement of Lemma~\ref{lem:EhLalpha}. The result follows.
\end{proof}


\noindent {\em Acknowledgements.} 
The authors are grateful to Julien Br\'emont for pointing out a mistake in a first version of this paper, and to the anonymous referee for the careful reading and many valuable remarks.


\end{document}